\documentclass[11pt]{article}
\usepackage{amsmath,amsthm,amsfonts,amssymb,mathrsfs,bm,graphicx,stmaryrd}
\usepackage{mathtools}
\usepackage[usenames]{color}
\usepackage[colorlinks=true,linkcolor=blue]{hyperref}
\usepackage[nottoc]{tocbibind}
\usepackage[letterpaper,hmargin=1.0in,vmargin=1.0in]{geometry}
\parskip 	\smallskipamount
\usepackage{tocstyle}

\newtheorem{theorem}{Theorem}[section]
\newtheorem{lemma}[theorem]{Lemma}
\newtheorem{corollary}[theorem]{Corollary}
\newtheorem{proposition}[theorem]{Proposition}

\newtheorem{remark}[theorem]{Remark}

\newtheorem{maintheorem}{Theorem}

\def\ind{{\mathbf 1}}
\def\N{\mathbb{N}}

\def\P{\mathbb{P}}
\def\Z{\mathbb{Z}}
\def\R{\mathbb{R}}

\def\E{\mathbb{E}}
\def\l{\ell}

\newcommand{\cZ}{\mathcal{Z}}

\newcommand{\Addresses}{{
  \bigskip
  \footnotesize

  Riddhipratim Basu, \textsc{International Centre for Theoretical Sciences, Tata Institute of Fundamental Research, Bangalore, India}\par\nopagebreak
  \textit{E-mail address}: \texttt{rbasu@icts.res.in}

  \medskip

  Sourav Sarkar, \textsc{Department of Statistics, University of California, Berkeley}\par\nopagebreak
  \textit{E-mail address}: \texttt{souravs@berkeley.edu}

  \medskip

  Allan Sly, \textsc{Department of Mathematics, Princeton University, Princeton, NJ, USA}\par\nopagebreak
  \textit{E-mail address}: \texttt{allansly@princeton.edu}

}}

\begin{document}
\title{Coalescence of Geodesics in Exactly Solvable Models of Last Passage Percolation}

\author{Riddhipratim Basu
 \and
Sourav Sarkar
\and
Allan Sly
}

\date{}
\maketitle

\begin{abstract}
Coalescence of semi-infinite geodesics remains a central question in planar first passage percolation. In this paper we study finer properties of the coalescence structure of finite and semi-infinite geodesics for exactly solvable models of last passage percolation. Consider directed last passage percolation on $\Z^2$ with i.i.d.\ exponential weights on the vertices. Fix two points $v_1=(0,0)$ and $v_2=(0, \lfloor k^{2/3} \rfloor)$ for some $k>0$, and consider the maximal paths $\Gamma_1$ and $\Gamma_2$ starting at $v_1$ and $v_2$ respectively to the point $(n,n)$ for $n\gg k$. Our object of study is the point of coalescence, i.e., the point $v\in \Gamma_1\cap \Gamma_2$ with smallest $|v|_1$. We establish that the distance to coalescence $|v|_1$ scales as $k$, by showing the upper tail bound $\P(|v|_1> Rk) \leq R^{-c}$ for some $c>0$. 

We also consider the problem of coalescence for semi-infinite geodesics. For the  almost surely unique semi-infinite geodesics in the direction $(1,1)$ starting from $v_3=(-\lfloor k^{2/3} \rfloor , \lfloor k^{2/3}\rfloor)$ and $v_4=(\lfloor k^{2/3} \rfloor ,- \lfloor k^{2/3}\rfloor)$, we establish the optimal tail estimate $\P(|v|_1> Rk) \asymp R^{-2/3}$, for the point of coalescence $v$. This answers a question left open by Pimentel \cite{Pim16} who proved the corresponding lower bound. 
%
%
%
%
\end{abstract}

\tableofcontents

\section{Introduction}
In their seminal paper in 1986, Kardar, Parisi, and Zhang \cite{KPZ86} predicted universal scaling behaviour for a large number of planar random growth processes, including first passage percolation and corner growth processes. KPZ scaling predicts that these models have length fluctuation exponent of $1/3$ and transversal fluctuation exponent $2/3$, although rigorous progress has been made only in a handful of cases. The first breakthrough was made by Baik, Deift and Johansson \cite{BDJ99} when they established $n^{1/3}$ fluctuation on the length of the longest increasing path from $(0,0)$ to $(n,n)$ in a homogeneous Poissonian field on $\R^2$ where they also established the GUE Tracy-Widom scaling limit. Transversal fluctuation exponent of $2/3$ in this model was proved by Johansson \cite{J00} who also proved  the $n^{1/3}$ fluctuation and Tracy-Widom scaling limit in directed last passage percolation on $\Z^2$ with i.i.d.\ passage times distributed according to either Geometric or  Exponential distribution~\cite{Jo99}. These are the exactly solvable models, for which many exact distributional formulae are available, typically using some deep machinery from algebraic combinatorics or random matrix theory, and certain duality properties from queueing theory in some cases. Over the last twenty years there has been tremendous progress in achieving a detailed understanding in these and a handful of other exactly solvable models, and a rich limiting theory has emerged; see \cite{Corwin11} for an excellent survey of this line of works. Understanding of the pre-limiting models, however, has remained mostly restricted to the exactly solvable cases.

In another related, but separate direction of works, a lot of progress has been made in studying planar first passage percolation, another model believed to be in the KPZ universality class. In absence of exact formulae, the study of first passage percolation has relied mostly on a geometric understanding of the geodesics. Although much less is rigorously known, the connection between understanding properties of infinite geodesics, limit shapes and the KPZ predicted fluctuation exponents has been clear for some years. Coalescence of geodesics has been an interesting tool to study the geometry of first passage percolation model, the study of which was initiated by Newman and co-authors as summarised in his ICM paper~\cite{New95} which proved certain coalescence results under curvature assumptions on the limit shape. Much progress has been made in recent years in understanding the geodesics starting with the breakthrough idea of Hoffman~\cite{H08} of studying infinite geodesics using Busemann functions. These techniques have turned out to be extremely useful, providing a great deal of geometric information on the structure of geodesics in first passage percolation~\cite{DH14, AH16, DH17}.

{In recent years there has been a great deal of interest in studying the coalescence of polymers (maximal paths which we shall also refer to as geodesics) in last passage percolation models as well \cite{FP05,C11, Pim16}. Much can be established in certain exactly solvable settings including the existence and uniqueness of semi-infinite geodesics starting at a given point along a given direction, and coalescence of geodesics along deterministic directions. Some of these results have recently been proved beyond exactly solvable models as well \cite{GRS15,GRS15+}. In this paper, we shall restrict ourselves to the exactly solvable setting of Exponential directed last passage percolation on $\Z^2$, and establish the precise order of the distance to coalescence for two semi-infinite geodesics along the same direction started at distinct points (see Theorem \ref{t:coalopt}) with the optimal tail estimate answering an open question from \cite{Pim16} who proved the corresponding lower bound. We, however, are also interested in the finite variants of the question, where we consider distance to coalescence of geodesics from two distinct points to a far away point. This variant is more important for some applications. We prove a similar scaling in this finite setting also (see Theorem \ref{t:coal}), albeit with a worse tail estimate. Our arguments combine moderate deviations from the exactly solvable literature with tools from percolation to understand geometry of a geodesic together with the environment around it. By way of the proof of this main result we also obtain a local transversal fluctuation result for the geodesics in last passage percolation (see Theorem \ref{t:carsestimate}) that is of independent interest. We now move towards precise model definition and the statement of the main results.}

\subsection{Model Definition and main results}
\label{s:lpp}

Consider the following last passage percolation (LPP) model on $\Z^2$. For each vertex $v\in \Z^2$ associate i.i.d.\ weight $\xi_{v}$ distributed as $\mbox{Exp}(1)$. Define $u\preceq v$ if $u$ is co-ordinate wise smaller than $v$ in $\Z^2$.
For any oriented path $\gamma$ from $u$ to $v$ let the passage time of $\gamma$ be defined by
$$\ell(\gamma):=\sum_{v'\in \gamma\setminus \{v\}} \xi_{v'}.$$
For $u\preceq v$ define the last passage time from $u$ to $v$, denoted $T_{u,v}$ by $T_{u,v}:=\max_{\gamma} \ell (\gamma)$ where the maximum is taken over all up/right oriented paths from $u$ to $v$. Let $\Gamma_{u,v}$ denote the (almost surely unique) path between $u$ and $v$ that attains the last passage time $T_{u,v}$. We shall call the path $\Gamma_{u,v}$ the geodesic between $u$ and $v$. \footnote{Observe that this is a little different from the usual definition of last passage percolation as we exclude the final vertex while adding weights. This is done for convenience as our definition allows $\ell(\gamma)=\ell(\gamma_1)+\ell(\gamma_2)$ where $\gamma$ is the concatenation of $\gamma_1$ and $\gamma_2$. As the difference between the two definitions is minor while considering last passage times between far away points, and the geodesics are same, all our results will be valid for both our and the usual definition of LPP.}

\subsubsection{Coalescence of finite geodesics}
{We now proceed towards statements of our main results. We first state our result in the finite setting. Let $\mathbf{n}$ denote the point $(n,n)$. Let $k>0$ be arbitrary and let $v_1=(0,0)$ and $v_2=(0,k^{2/3})$ (assume without loss of generality that $k^{2/3}$ is an integer, the same result holds with $\lfloor k^{2/3} \rfloor$ otherwise). Let $v_{*}=(v_{*,1}, v_{*,2})$ be a leftmost common point between $\Gamma_{v_1, \mathbf{n}}$ and $\Gamma_{v_2, \mathbf{n}}$ (observe that $v_{*,1}$ is well-defined even if there is no unique leftmost common vertex). Our main result in this paper shows that $v_{*,1}$ is of order $k$ when $n\gg k$.}

\begin{maintheorem}
\label{t:coal}
There exist positive constants $R_0,C,c>0$ such that for all $R>R_0$ and for all $k>0$ the coalescence location satisfies
$$\limsup_{n\to \infty} \P(v_{*,1}> Rk)\leq CR^{-c}.$$
\end{maintheorem}

It is natural to predict this scaling from the KPZ transversal fluctuation exponent of $2/3$ which says that the geodesic between two points at distance $n$ fluctuates at scale $n^{2/3}$ away from the straight line joining the two points. However, to prove Theorem~\ref{t:coal}, we need finer local control on the transversal fluctuation of the geodesic (see Theorem \ref{t:carsestimate}) below.

This coalescence result is robust in the following ways. We do not need to consider geodesics to have the same endpoint, merely that their distance is at the correct scale of transversal fluctuation. Thus Theorem \ref{t:coal} will be valid for the leftmost common point of $\Gamma_{v_1,\mathbf{n}}$ and $\Gamma_{v_2, (n,n+n^{2/3})}$ as well. Furthermore, the choice of direction is arbitrary. The same result holds for geodesics to $(n,hn)$ for any fixed $h\in(0,\infty)$. See Corollary \ref{c:pathmeetslope} for a precise statement. Finally, it would also be clear from the proof that we need not have taken $v_1$ and $v_2$ on a vertical line; the same proof would have worked on two points at distance $k^{2/3}$ on the line $x+y=0$, say. As a matter of fact, the same proof will show that a similar tail estimate works for $v_{*,2}$ and hence for $|v_{*}|_1$ as well, as claimed before. 

We also mention that we work with the Exponential last passage percolation merely for concreteness. Our proof depends only on the Tracy-Widom limit and one point upper and lower tail moderate deviation estimates for the last passage times (see Theorem \ref{t:Jo99} and Theorem \ref{t:moddevdiscrete}) and should work equally well for other exactly solvable models where such estimates are available. Indeed such estimates are available for Poissonian directed last passage percolation in continuum \cite{LM01, LMS02} and last passage percolation on $\Z^2$ with geometric passage times \cite{BXX01, CLW15}, and variants of our results should apply to those models as well.


\subsubsection{Semi-infinite geodesics}
\label{s:bg}
As already mentioned for the case of semi-infinite geodesics, one can obtain a more precise asymptotic result. Before a statement of the result in a proper context, we need to introduce a few definitions and develop some background. The study of semi-infinite geodesics in last passage percolation with Exponential passage times was initiated in \cite{FP05,C11} where the following general picture was established. Starting from any $x\in \Z^2$ there exists an almost surely unique semi-infinite path $\Gamma_{x}=\{x=x_0,x_1,x_2,\ldots\}$ such that for each $i<j$ the section of $\Gamma_x$ between $x_i$ and $x_j$ is the geodesic between $x_i=(x_{i,1},x_{i,2})$ and $x_j$, and such that $\lim_{n\to \infty}\frac{x_{n,1}}{x_{n,2}}=1$. Such a path is called the semi-infinite geodesic starting at $x$ in direction $(1,1)$.  Moreover, any sequence of finite geodesics from $x$ to points $y_n$ in the asymptotic direction $(1,1)$ converges to $\Gamma_x$ almost surely. Finally, this collection of semi-infinite geodesics $\{\Gamma_{x}\}_{x\in \Z^2}$ almost surely coalesce, i.e., for any $x, x'\in \Z^2$, the number of vertices in $\Gamma_{x} \Delta \Gamma_{x'}$ is finite ($\Delta$ denotes the symmetric difference between the two sets of vertices). The same result holds for any positive quadrant direction bounded away from the coordinate axial directions.

This set of results closely parallels the results of Newman and co-authors in early 90s as summarized in \cite{New95} in the context of first passage percolation under certain assumptions on curvature of the limit shape. In recent years the coalescence structure of semi-infinite geodesics has become a central object of study and a lot of progress has been made using Busemann functions \cite{H08, AH16,DH17, DH14}. However, with the help of integrable structure much finer results can be established in the LPP setting with exponential weights.

Distance to coalescence for semi-infinite geodesics along the same direction is a natural object of study in the integrable setting and was considered in \cite{Pim16}, and a similar scaling was predicted. Using Burke's duality, and Busemann functions \cite{Pim16} established, among other things, a lower bound to this effect, and the upper bound remained open. For technical convenience, let us change the setting of Theorem \ref{t:coal} slightly. 

Fix $k\in \N$. Consider the straight line $\mathbb{L}=\{(x,y)\in \Z_2: x+y=0\}$. on $\mathbb{L}$ (assume, as before, without loss of generality that $k^{2/3}$ is an integer). For $v$ on the line $\mathbb{L}$, let $\Gamma_{v}$ denote the almost surely unique semi-infinite geodesic starting at $v$ in the direction $(1,1)$. Recall that the collection $\{\Gamma_{v}\}_{v\in \mathbb{L}}$ is coalescing, i.e., for  any $v,v'\in \mathbb{L}$, almost surely $\Gamma_{v}$ and $\Gamma_{v'}$ coalesce. Let $c(v,v')=(x(v,v'),y(v,v'))$ denote the point at which $\Gamma_{v}$ and $\Gamma_{v'}$ coalesce. Let $d(v,v')=x(v,v')+y(v,v')$ denote the distance to coalescence. Now consider $v_3=(-k^{2/3}, k^{2/3})$ and $v_4=(k^{2/3},-k^{2/3})$ (assume without loss  of generality that $k^{2/3}\in \N$). Translated to this setting, \cite{Pim16} proved that $\lim_{R\to 0} \P(d(v_3,v_4)>Rk) \to 1$ as $R\to 0$ uniformly in $k$, further it was conjectured that, this is the correct scaling, i.e., $\lim_{R\to 0} \P(d(v_3,v_4)>Rk) \to 0$ as $R\to \infty$. Moreover, using some calculations in the limiting Airy process, \cite{Pim16} conjectured that $\P(d(v_3,v_4)>Rk)\asymp R^{-2/3}$ as $R\to \infty$ uniformly in $k$. Our second main result settles this conjecture. 

\begin{maintheorem}
\label{t:coalopt}
In the above set-up, there exists $C_1,C_2, R_0>0$ such that for all $k>0$ and $R>R_0$, we have 
$$C_1R^{-2/3} \leq \P(d(v_3,v_4)>Rk)\leq C_2R^{-2/3}.$$
\end{maintheorem}

As a matter of fact a lower bound to this effect was already proved in \cite{Pim16} (see \cite[Section 3]{Pim16} for a discussion of the difficulty of the approach therein to get a matching upper bound), hence the main work goes in proving the upper bound. However we shall also provide a short proof for the lower bound for completeness. 

{Note that we could obtain only a polynomial decay of an unspecified small exponent in Theorem \ref{t:coal}(with much more work). This is because in the infinite setting, we can appeal to certain ergodic theorems which end up giving tight results, but seem hard to adapt to the finite setting.}

Observe also that we took the points $v_3$ and $v_4$ on the anti-diagonal line $x+y=0$ rather than on a vertical line as in the set-up of Theorem \ref{t:coal}. It is merely for technical convenience, the same result will still hold for points on the vertical line, as well as for semi-infinite geodesics in other directions (except the axial directions) with minor changes to the proof.

We finish this subsection with a further discussion about the importance of studying the coalescence of finite geodesics in models of KPZ universality class. The coalescence structure of geodesics in exactly solvable polymer models is important in understanding scaling limits of the random geometric structures. See \cite{Pim16} for connections to this question to a conjectural object the Airy Sheet. As mentioned above, in Brownian last passage percolation, where one has a strong resampling property called Brownian Gibbs property \cite{CH14}, a much more detailed structure of coalescent polymer trees has been explored and used to make progress towards the important question of Brownian regularity of the Airy processes \cite{H16, H17a, H17b, H17c}. Using techniques of \cite{Pim16}, local Brownian regularity has also been explored in \cite{Pim17}. 

Finally we advocate a further reason for studying the geometry of geodesics in the context of exactly solvable polymer models. A detailed understanding of the geometry of geodesics, beyond what can be obtained from the integrable techniques have recently proved useful in study of certain models that arise from adding local defects to the integrable models. Even though the local defects destroy the integrable structure, the more geometric understanding of the geodesics are still useful. One example of this principle was obtained in \cite{BSS14}, where TASEP with a slow bond at the origin was studied using the correspondence to Exponential last passage percolation, and geometric understanding of the geodesic was used to show that a slow bond at the origin of arbitrarily small strength changes the current, thus settling the ``slow bond problem". As a matter of fact, we use the results in this paper to study the invariant measures of TASEP with a slow bond \cite{BSS17+}, and establish a conjecture of Liggett from \cite{Lig99}. We remark that for these applications, it is crucial to have the coalescence statement in the finite setting. Hence even though we do not get the optimal tail decay in Theorem \ref{t:coal}, the result turns out to be important and useful.

\subsection{Inputs from Integrable Probability; Tracy-Widom limit, and $n^{2/3}$ Fluctuations}
\label{s:int}

In this subsection, we recall the basic inputs from the integrable probability literature that we shall be using throughout. As mentioned before Exponential DLPP is one of the handful of models for which the KPZ scaling result and much more has been rigorously established. The Tracy-Widom scaling limit for exponential DLPP is due to Johansson \cite{Jo99}.

\begin{theorem}[\cite{Jo99}]
\label{t:Jo99}
Let $h>0$ be fixed. Let $v=(0,0)$ and $v_{n}=(n,\lfloor hn \rfloor)$. Then
\begin{equation}
\label{e:twlimitdiscrete}
\dfrac{T_{v,v_n}-(1+\sqrt{h})^2n}{h^{-1/6}(1+\sqrt{h})^{4/3}n^{1/3}} \stackrel{d}{\rightarrow} F_{TW}.
\end{equation}
where the convergence is in distribution and $F_{TW}$ denotes the GUE Tracy-Widom distribution.
\end{theorem}

GUE Tracy-Widom distribution is a very important distribution in random matrix theory that arises as the scaling limit of largest eigenvalue of GUE matrices; see e.g.\ \cite{BDJ99} for a precise definition of this distribution. For our purposes moderate deviation inequalities for the centred and scaled variable as in the above theorem will be important. Such inequalities can be deduced from the results in \cite{BFP12}, as explained in \cite{BSS14}. We quote the following result from there.

\begin{theorem}[\cite{BSS14}, Theorem 13.2]
\label{t:moddevdiscrete}
Let $\psi>1$ be fixed. Let $v, v_n$ be as in Theorem \ref{t:Jo99}. Then  there exist constants $N_0=N_0(\psi)$, $t_0=t_0(\psi)$ and $c=c(\psi)$ such that we have for all $n>N_0, t>t_0$ and all $h \in (\frac{1}{\psi}, \psi)$

$$\P[|T_{v,v_{n}}- n(1+\sqrt{h})^{2}|\geq tn^{1/3}]\leq e^{-ct}.$$
\end{theorem}

Theorem \ref{t:moddevdiscrete} provides much information about the geometry and regularity of geodesics in the DLPP model; which was exploited crucially in \cite{BSS14}, and will be extensively used by us again. Most fundamental among those is the $n^{2/3}$ transversal fluctuation of the geodesic between points at distance $n$. Let $\gamma$ denote the (almost surely unique) geodesic between two fat away points $u$ and $v$. For simplicity let us assume $u=(0,0)$ and $v=(n,n)$. The transversal fluctuation of $\gamma$, denoted $\mbox{TF}_{n}$ is defined by $\sup_{(x,y)\in \gamma} |x-y|$ and denotes the maximum vertical distance between a point on $\gamma$ to the straight line joining the two points. It follows from Theorem \ref{t:moddevdiscrete} that $\mbox{TF}_{n}$ is an order $n^{2/3}$ random object. The scaling exponent $2/3$ was identified in \cite{J00};
it follows from the arguments of \cite{BSS14},  cf. Theorem 11.1 there (see also Theorem 2.5 in \cite{BCS06} for an argument using Burke's duality) that $\P({\rm TF}_{n}> kn^{2/3})\leq e^{-ck^2}$ uniformly in large $n$ for some $c>0$.


Observe that Theorem 11.1 in \cite{BSS14} provides a global upper bound on transversal fluctuation. However, from points $(x,y)\in \gamma$  with $x=\ell \ll n$, one expects a much smaller transversal fluctuation of order $\ell^{2/3}$. Such a local fluctuation estimate would be useful to us and is also of independent interest. We now move towards a precise statement to this effect.

Let $\Gamma$ be the geodesic from $(0,k')$ to $(n,n+k)$.  For $\ell\in \Z$, let $\Gamma(\ell)\in \Z$ be the maximum number such that $(\ell,\Gamma(\ell))\in \Gamma$ and $\Gamma^{-1}(\ell)\in \Z$ be the maximum number such that $(\Gamma^{-1}(\ell),\ell)\in \Gamma$. The following theorem is our final main result in this paper.


\begin{maintheorem}
\label{t:carsestimate}
Fix $L>0$. Then there exist  positive constants $n_0,\l_0, s_0, c$ depending only on $L$, such that for all $n\geq n_0, s\geq s_0 \vee 2L, \l\geq \l_0$, and $|k'|\leq L\ell^{2/3}, |k|\leq Ln^{2/3}$ we have
\[\P[|\Gamma(\ell)-\ell|\geq s\ell^{2/3}] \leq e^{-cs^2};\]
\[\P[|\Gamma^{-1}(\ell)-\ell |\geq s\ell^{2/3}] \leq e^{-cs^2}.\]
\end{maintheorem}

\begin{remark}
It will be clear from the proof that the exponent $s^2$ here is determined by the exponents in the moderate deviation tail estimates. In the Poissonian LPP, where optimal moderate deviation bounds are known \cite{LM01, LMS02}, one can improve the bound further to $e^{-cs^3}$, which is optimal.
\end{remark}

{This result is of independent interest as it provides information on local transversal fluctuation of the geodesics, and has already been useful in several different contexts. For example, this has been used to study the locally Brownian nature of the pre-limiting Airy process profile for Exponential LPP on short scales and to study the time correlation of the same \cite{BG18}, and also the modulus of continuity for poylmer fluctuations and weight profiles in Poissonian LPP \cite{HS18}. A variant of this estimate for first passage percolation can be used to control the amount of backtrack in the geodesics \cite{BSS17}. For our purposes here we shall need a more refined version of this estimate; see Theorem \ref{t:carsbetter} below.} 

A variant of this result also holds for semi-infinite geodesics; see Proposition \ref{p:carinfinite}. We shall use this result to prove Theorem \ref{t:coalopt}. Using Proposition \ref{p:carinfinite}, one can also provide an alternative proof of the lower bound in \cite{Pim16} (see Remark \ref{r:lb}).

\subsection{Outline of the Proof of Theorem \ref{t:coal}}
\label{s:outline}
We describe now the basic outline of our proof of Theorem \ref{t:coal}. For some large fixed number $M$, we try to achieve coalescence at length scales $M^{i}k$ for different values of $i$. We show that at each scale coalescence happens with probability bounded below independent of $i$, and these events are approximately independent, i.e., failure to coalesce at one scale does not make coalescence at the next scale much less likely. Trying at a large number of length scales one obtains Theorem \ref{t:coal}. More precisely we establish the following.

Let $L$ be a large fixed number. For $r\in \N$ consider the points $a_{r}=(r,r+Lr^{2/3})$ and $b_{r}=(r,r-Lr^{2/3})$ (assume without loss of generality that $r^{2/3}$ is an integer). For some large $M\in \N$ let $\Gamma_1$ denote the geodesic from $a_{r}$ to $a_{Mr}$, and $\Gamma_2$ denote the geodesic from $b_r$ to $b_{Mr}$. Let ${\rm Coal}_{r,M}$ denote the event that $\Gamma_1$ and $\Gamma_2$ share a common vertex. We have the following theorem.

\begin{theorem}
\label{t:finite1}
Fix $L>0$. Then for all sufficiently large $M$ and all $r>0$ we have
$$\P({\rm Coal}_{r,M})\geq \frac{1}{2}.$$
\end{theorem}

Theorem \ref{t:finite1} says that if the geodesics $\Gamma_{v_1, \mathbf{n}}$ and $\Gamma_{v_2, \mathbf{n}}$ from Theorem \ref{t:coal} does not have an atypically high transversal fluctuation at distances $r$ and $Mr$, then with probability at least $\frac{1}{2}$, they coalesce in $[r,Mr]$. Observe that Theorem \ref{t:carsestimate} says that atypical transversal fluctuation at a given point is exponentially unlikely, later we establish a refinement of this showing that such events are also roughly independent if $M$ and $r$ are sufficiently large (see Theorem \ref{t:carsbetter}), thereby establishing Theorem \ref{t:coal}. 


Most of the work in this paper goes into the proof of Theorem \ref{t:finite1}. This is done via a bootstrapping argument. We first show that the probability is bounded below by an arbitrary small constant independent of $r$ and $M$ (see Proposition \ref{l:meetsubinter}).
This follows from showing at some horizontal length scale distance $D$ (where $r\ll D \ll Mr$) with a probability bounded away from zero there exists a barrier of width $O(D)$ and height $O(D^{2/3})$ just above the geodesic $\Gamma_{2}$, such that any path passing through the barrier is penalised a lot. Using Theorem \ref{t:carsestimate} and the FKG inequality we show that in presence of such a barrier and in environment that is typical otherwise, $\Gamma_{2}$ will merge with $\Gamma_{1}$ before crossing the barrier region. The construction of the barrier here is similar to one present in \cite{BSS14}, and we shall quote many of the probabilistic estimates in that paper throughout our proof.

We note here that events forcing coalescence of finite and infinite geodesics have been studied in a number of works in different settings. Some of these are done using ergodicity in non-integrable settings and are inherently not quantitative \cite{New95,LN96,AH16}, while the preprint \cite{DPM17} considers a rectangle of size $n^{1+o(1)}\times n^{2/3}$ and show that best paths constrained to stay within this rectangle coalesce with rather weak probability lower bound of $n^{-o(1)}$ \footnote{This was recently brought to our attention by Ron Peled, over a year after the first version of this paper was posted on arXiv. The coalescence estimate in \cite{DPM17} nevertheless is sufficient for their purpose of establishing a central limit theorem for paths constrained to be in an off-scale thin rectangle.}.  Our approach of further developing the combination of geometric techniques and integrable inputs, introduced in \cite{BSS14}, together with the control on local fluctuations of geodesics (Theorem \ref{t:carsestimate}), in contrast, leads to a proof that coalescence happens with uniformly positive probability at the correct length scale.

\subsection{Notations}
For easy reference purpose, let us collect here a number of notations, some of which have already been introduced, that we shall use throughout the remainder of this paper. Define the partial order $\preceq $ on $\Z^2$ by $u = (x,y) \preceq u' = (x',y')$ if $x \leq x'$, and $y \leq y'$. For $a,b\in \Z^2$ with $a\preceq b$, let $\Gamma_{a,b}$ denote the geodesic from $a$ to $b$ in the Exponential LPP, and $T_{a,b}$ denotes the weight of the geodesic $\Gamma_{a,b}$.

For an increasing path $\gamma$ and $\ell\in \Z$, $\gamma(\ell)\in \Z$ will denote the maximum number such that $(\ell,\gamma(\ell))\in \gamma$ and $\gamma^{-1}(\ell)\in \Z$ be the maximum number such that $(\gamma^{-1}(\ell),\ell)\in \gamma$. Let $\l(\gamma)$ denote the weight of the increasing path $\gamma$. Also for $a<b<c<d \in \Z$, and $\gamma$ an increasing path from $(a,a')$ to $(d,d')$, we define \[\gamma[b,c]=\{\gamma(x):b\leq x\leq c\}\]
 as the part of $\gamma$ between the vertical lines $x=b$ and $x=c$.

 For $u=(x,y) \preceq u'=(x',y')$ in $\Z^2$, let $d(u,u')=(x'-x)+(y'-y)$ denote the $\l_1$ distance between $u$ and $u'$. Define
\[\widetilde{T}_{u,u'}=T_{u,u'}-\E(T_{u,u'}),\]
\[\widehat{T}_{u,u'}=T_{u,u'}-2d(u,u').\]
It is easy to see that $\widehat{T}_{u,u'}\leq \widetilde{T}_{u,u'}$. Roughly speaking, if the slope of the line joining $u$ and $u'$ is close to $1$, then $\widetilde{T}_{u,u'}$ can be well approximated by $\widehat{T}_{u,u'}$ (see Section 9 of \cite{BSS14}).

Also for any set $R\subseteq \R^2$, let $T^R_{u,v}$ denote the weight of the maximal path from $u$ to $v$ that avoids the region $R$. Let $^RT_{u,v}$ denote the weight of the maximal path from $u$ to $v$ that intersects the region $R$. Also define $\widetilde{T}^R_{u,v}=T^R_{u,v}-\E T_{u,u'}$ and $\widehat{T}^R_{u,v}=T^R_{u,v}-2d(u,u')$. Similarly define $^R\widetilde{T}_{u,v}$ and $^R\widehat{T}_{u,v}$.

We shall use the notation $\llbracket \cdot, \cdot \rrbracket$ to denote discrete intervals, i.e., $\llbracket a,b \rrbracket$ will denote $[a,b]\cap \Z$. We shall often assume without loss of generality that fractional powers of integers i.e., $k^{2/3}$ or rational multiples of integers as integers themselves. This is done merely to avoid the notational overhead of integer parts, and it is easy to check that such assumptions do not affect the proofs in any substantial way. In the various theorems and lemmas, the values of the constants $C,C',c,c'$ appearing in the bounds change from one line to the next, and will be chosen small or large locally.

\subsection{Organisation of the paper}
The rest of the paper is organised as follows. In Section \ref{s:cars} we prove Theorem \ref{t:carsestimate} and also prove a refinement Theorem \ref{t:carsbetter}. In Section \ref{s:coales},  we start prove Theorem  \ref{t:finite1} by reducing it to the key Proposition \ref{l:meetsubinter}, and use it to establish Theorem \ref{t:coal}. The next two sections are devoted to the proof of Proposition \ref{l:meetsubinter}. In Section \ref{s:events}, we define a  geometric structure and a number of key events used in the rest of the proof. Section \ref{s:const} constructs a collection of coalescing paths on a combination of the key events and estimates the corresponding probabilities, and concludes the proof of Proposition \ref{l:meetsubinter}. Section \ref{s:coalopt} contains the proof of Theorem \ref{t:coalopt}. This final section is independent of the rest of the paper except that we use a variant of Theorem \ref{t:carsestimate}.

\subsection*{Acknowledgements}
RB thanks Alan Hammond for useful discussions and explaining his work on coalescent polymer trees in Brownian last passage percolation, and Christopher Hoffman for asking a question which led to Theorem \ref{t:coalopt}. He also thanks Vladas Sidoravicius for many useful discussions, and Ron Peled for pointing out the related work \cite{DPM17}. RB was partially supported by an AMS-Simons Travel Grant during the early phases of this project and is partially supported by an ICTS Simons Junior Faculty Fellowship and a Ramanujan Fellowship from Govt. of India. SS is supported by a Lo\'{e}ve Fellowship. AS is supported by NSF grant DMS-1352013 and a Simons Investigator grant. Part of this research was performed during two visits of RB to the Princeton Mathematics department, he gratefully acknowledges the hospitality. 


\section{A path regularity estimate}
\label{s:cars}
Our objective in this section is to prove Theorem \ref{t:carsestimate} and prove a refinement of the same. The proof of Theorem \ref{t:carsestimate} is reminiscent of an argument in \cite{New95} in the context of first passage percolation, however is much stronger than the result there as for first passage percolation one has much weaker information about fluctuation of passage times and transversal fluctuations.




\begin{proof}[Proof of Theorem \ref{t:carsestimate}]
We show only that $\P[(\Gamma(\l)-\l)\geq sr^{2/3}] \leq e^{-cs^2}$. The other parts are similar. Also for convenience, we assume that $k'=0$, so that $\Gamma$ is the geodesic from $(0,0)$ to $(n,n+k)$. For general $k'\leq L\l^{2/3}$, the argument is similar.

Choose $\alpha=2^{\frac{1}{6}}$.
For $j\geq 0$, let $B_j$ denote the event that $\Gamma(2^j\l)-2^j\l\geq s((2\alpha)^j\l)^{2/3}$ and $\Gamma(2^{j+1}\l)-2^{j+1}\l\leq s((2\alpha)^{j+1}\l)^{2/3}$. Note that as $|\Gamma(n)-n|=|k|\leq Ln^{2/3}<sn^{2/3}$ for all $s\geq 2L$, hence,
\[\{\Gamma(\l)-\l\geq s\l^{2/3}\}\subseteq \bigcup_{j\geq 0} B_j.\]
See Figure \ref{f:cars}.
\begin{figure}[h]
\centering
\includegraphics[width=0.4\textwidth]{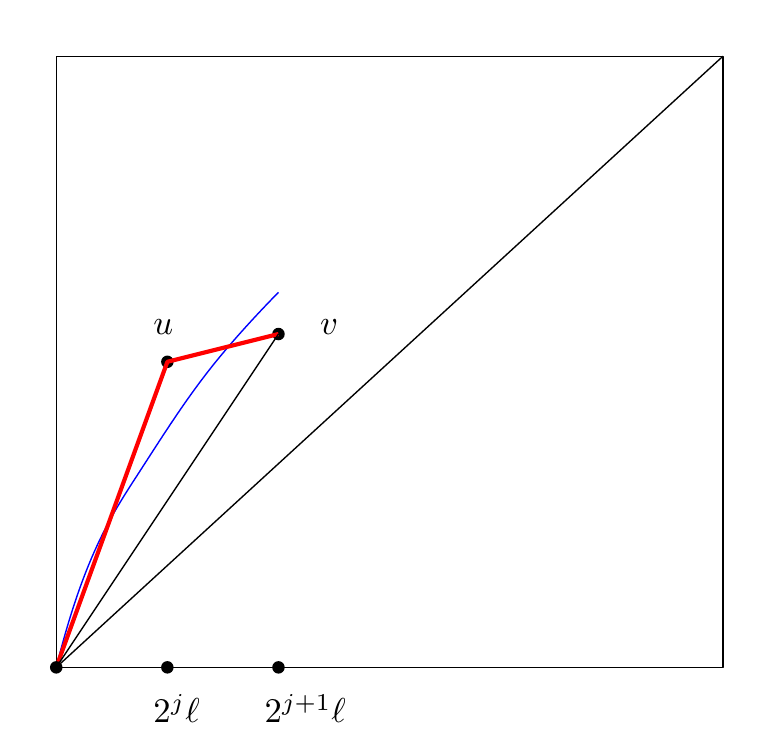}
\caption{Proof of Theorem \ref{t:carsestimate}: the blue curve in the figure is the graph of the function $y=x+sx^{2/3}$.  On the event $\Gamma(\ell)-\ell\geq s\ell^{2/3}$, there must exist some $j\geq 0$ such that $\Gamma$ crosses the blue curve on the interval $[2^j\ell, 2^{j+1}\ell]$. The event $B_j$ is a slightly more involved variant of the event described above. On the event $B_{j}$ one must have points $u$ and $v$ as above such that the path from $(0,0)$ to $u$ and then to $v$ is atypically long, hence this event is unlikely.}
\label{f:cars}
\end{figure}

Hence it suffices to show that
\[\P(B_j)\leq e^{-cs^2\alpha^{2j/3}}.\]
Let $B_{j,t,t'}$ denote the event that
\[\Gamma(2^j\l)\in U_t:=\left[2^j\l+(s+t)((2\alpha)^j\l)^{2/3},2^j\l+(s+t+1)((2\alpha)^j\l)^{2/3}\right]\]
and
\[\Gamma(2^{j+1}\l)\in V_{t'}:= \left[2^{j+1}\l+(s-t')((2\alpha)^{j+1}\l)^{2/3},2^{j+1}\l+(s-t'+1)((2\alpha)^{j+1}\l)^{2/3}\right].\]
for $t,t'=0,1,2,\ldots$.  Clearly,
\[B_{j,t,t'}\subseteq \{\sup_{u\in U_t,v\in V_{t'}}( T_{0,(2^j\l,u)}+T_{(2^j\l,u),(2^{j+1}\l,v)}-T_{0,(2^{j+1}\l,v)})\geq 0\}.\]

If $\mathcal{S}$ is the line segment joining $\mathbf{0}$ to some vertex $v\in V_{t'}$, then it is easy to see that
\[\mathcal{S}(2^j\l)-2^j\l\leq \frac{\alpha^{2/3}}{2^{1/3}}(s-t'+1)((2\alpha)^{j}\l)^{2/3}.\]
Thus computing expectations, it follows from Lemma $9.4$ of \cite{BSS14} \footnote{Using similar observations as made in Section $9$ of \cite{BSS14} for Poissonian LPP, from Theorem \ref{t:Jo99} and Theorem \ref{t:moddevdiscrete}, it follows that, in Exponential LPP model, for fixed $\psi>0$, there exists $r_0=r_0(\psi)$ such that for points $u=(x,y)$ and $u'=(x',y')$ in $\Z^2$ such that  $x'-x=r\geq r_0$, and $\frac{y'-y}{x'-x}\in (\frac{1}{\psi},\psi)$, one has, $\E(T_{u,u'})=(\sqrt{r}+\sqrt{y'-y})^2+O(r^{1/3})=d(u,u')+2\sqrt{r(y'-y)}+O(r^{1/3})$. Hence, Corollary $9.3$, Lemma $9.4$ and Lemma $9.5$ from \cite{BSS14} continue to hold for Exponential LPP as well with the same proof.} and the fact that $2^{1/2}>\alpha>1$, that there exists some constant $c_1$ not depending on $\l,s,t,t',j$, such that for all $u\in U_t, v\in V_{t'}$, and all $s$ sufficiently large,
\[\E(T_{0,(2^j\l,u)})+\E(T_{(2^j\l,u),(2^{j+1}\l,v)})\leq \E(T_{0,(2^{j+1}\l,v)})-c_1((s+t+t')\alpha^{\frac{2j}{3}})^2(2^j\l)^{1/3}.\]
Using the moderate deviation estimates for supremum and infimum of the lengths of a collection of paths given in Proposition $10.1$ and Proposition $10.5$ of \cite{BSS14} (and breaking $U_t$ and $V_{t'}$ into consecutive intervals of length $(2^j\l)^{2/3}$ and taking $\alpha^{4j/3}$-many union bounds), and using similar arguments as in the proof of Lemma $11.3$ of \cite{BSS14}, this implies,
\[\P(B_{j,t,t'})\leq e^{-c(s+t+t')^2\alpha^{\frac{2j}{3}}}.\]
Summing over $t,t',j$ gives the result.
\end{proof}

\subsection{An improved regularity estimate}

Observe that Theorem \ref{t:carsestimate} says that at any given length scale $\ell$, the geodesic is unlikely to have a transversal fluctuation that is much larger than $\ell^{2/3}$. Our next result will show a decorrelation between these unlikely events at well separated length scales.

Let $\Gamma$ denote the geodesic from $\mathbf{0}$ to $\mathbf{n}$. For $k, M\in \Z$ fixed, and any vertex $v\in \Z^2$, let $A^v_{i}$ denote the event that $|\Gamma_{v,\mathbf{n}}(M^{i}k)-M^{i}k| \geq s (M^{i}k)^{2/3}$. Also let $A_i:=A_i^\mathbf{0}$. We have the following theorem.

\begin{theorem}
\label{t:carsbetter}
There exist positive constants $s_0, M_0, c,c'$ such that for all $s>s_0$, $M>M_0$ and $\ell,k\in \Z$ we have
$$\limsup_{n\to \infty} \P\left(\sum_{i=1}^{\ell} \ind_{A_{i}}\geq 2e^{-cs}\ell\right) \leq e^{-c'\ell}.$$
\end{theorem}

Theorem \ref{t:carsbetter} will follow from the next proposition.

\begin{proposition}
\label{t:allcars}
Let $F\subseteq [\l]$. Then there exist positive constants $c,s_0,M_0$ such that for all $s\geq s_0, M\geq M_0$,
\[\P(|\Gamma(M^ik)-M^ik|\geq s(M^ik)^{2/3} \mbox{ for all } i\in F)\leq e^{-cs|F|}.\]
\end{proposition}

We postpone the proof of Proposition \ref{t:allcars} for now, and first show how Theorem \ref{t:carsbetter} follows from Proposition \ref{t:allcars}.

\begin{proof}[Proof of Theorem \ref{t:carsbetter}]
Let $B_i$ s for $i\in [\l]$ be i.i.d. Bernoulli random variables with success probability $e^{-cs}$. Then Theorem \ref{t:allcars} implies that $(\ind_{A_1},\ind_{A_2},\ldots,\ind_{A_\l})\preceq _{ST}(B_1,B_2,\ldots,B_\l)$. Hence, $\sum_{i=1}^{\l}\ind_{A_i}\preceq_{ST}
\sum_{i=1}^{\l}B_i $. Theorem \ref{t:carsbetter} now follows from Hoeffding's inequality applied to $\sum_{i=1}^{\l}B_i$.
\end{proof}
{For the proof of Proposition \ref{t:allcars} we will need the following lemma that is basic and was stated in \cite{BSS14}. As we would have several occasions to resort to this lemma, we restate it here without proof.}

{\begin{lemma}[\cite{BSS14}, Lemma 11.2, Polymer Ordering]
\label{l:porder}
Consider points $a=(a_1,a_2),a'=(a_1,a_3),b=(b_1,b_2),b'=(b_1,b_3)$ such that $a_1 < b_1$ and $a_2\leq a_3\leq b_2\leq b_3$. Then we have $\Gamma_{a,b}(x)\leq \Gamma_{a',b'}(x)$ for all $x\in \llbracket a_1,b_1\rrbracket$.
\end{lemma}}

We shall now prove Proposition \ref{t:allcars}. Before starting with the technicalities of the proof let us explain the basic idea which is however simple. Consider the case $k=1$ and $F=\{1,2\}$. Using constructions as in the proof of Theorem \ref{t:carsestimate}, for $M$ sufficiently large we can approximate the event $A_1:=|\Gamma(M)-M|\geq sM^{2/3}$, up to a very small error in probability by an event $B_1$ that depends only on the random field on $\llbracket 0,D\rrbracket \times \Z$, for some $M\ll D \ll M^2$. The main point is that even on the unlikely event $A_1$, it is very likely that $|\Gamma(D)-D|\leq s' D^{2/3}$ for some $s'$ that is not too large. This implies the event $A_2:=|\Gamma(M^2)-M^2|\geq sM^{4/3}$ can then be well approximated by another unlikely event $B_2$ that is measurable with respect to the random field on $\llbracket D+1, n\rrbracket \times \Z$, and hence independent of $B_1$. The following proof makes this idea precise.

\begin{proof}[Proof of Proposition \ref{t:allcars}]
Without loss of generality we assume $k=1$. Also let $F=[\l]$. For any fixed subset $F\subseteq[\l]$, the proof follows similarly. Fix $\alpha<2^{1/2}$. Choose $C'$ large enough so that $\alpha^{2C'/3}\geq 2$. Let $M=2^C$ where $C$ is large enough such that $\alpha^{C'}<2^{C-C'}$. Let $A^{'v}_{i}$ denote the event
that $\Gamma_{v,\mathbf{n}}(M^{i}k)-M^{i}k \geq s (M^{i}k)^{2/3}$. Also let $A_i':=A_i^\mathbf{'0}$.

 Fix $\l,s$. For any $r\in [\l]$ and any $0<x\leq M^{r+1}$ and any $s_1\geq s$ such that $s_1((2\alpha)^{C'}x)^{2/3}<s(M^{r+2})^{2/3}$,   let $E_{r+1,x,s_1}=\{\Gamma(x)-x\geq s_1x^{2/3}\}$. Define $E_{r+1,x,s_1}^{v}$ similarly with the geodesic $\Gamma$ replaced by $\Gamma_{v,n}$. Also let $0\leq z<x$ be such that $x-z\geq \beta x$ for some fixed positive constant $\beta$, and let $s_2$ be such that $s_2z^{2/3}<\frac{s_1x^{2/3}}{2}$.
We claim that for any such $x$ fixed, and any such $z<x$ fixed and any such $s_1,s_2$, and for $v=(z,z+s_2(z)^{2/3})$
\begin{equation}\label{e:ind1}
\P\left(E^v_{r+1,x,s_1}\bigcap_{i\in \llbracket r+2,\l\rrbracket} A^{'v}_{i}\right)\leq 2^{\l-r}e^{-cs_1-cs(\l-r-1)},
\end{equation}
where $c$ is some absolute constant.

We prove the statement \eqref{e:ind1} by induction on $\l-r$. Clearly this holds when $r=\l-1$ by simply applying Theorem \ref{t:carsestimate}. Assume \eqref{e:ind1} holds for $r=k+1$, we prove this for $r=k$. Fix $x\leq M^{k+1}$, $z<x$ and $s_1\geq s$ such that $x-z\geq \beta x$ and $s_1((2\alpha)^{C'}x)^{2/3}<s(M^{k+2})^{2/3}$. Parallel to the events in Theorem \ref{t:carsestimate}, for $v=(z,z+s_2z^{2/3})$, define $B^v_j$ as the event that $\Gamma_{v,n}(2^jx)-2^jx\geq s_1((2\alpha)^jx)^{2/3}$ and $\Gamma_{v,n}(2^{j+1}x)-2^{j+1}x\leq s_1((2\alpha)^{j+1}x)^{2/3}$ for all $j=0,1,2,\ldots,C'-1$. Then
\begin{eqnarray*}
B&:=&E^v_{k+1,x,s_1}\bigcap_{i\in \llbracket k+2,\l\rrbracket} A^{'v}_{i}\\
&\subseteq & \bigcup_{j=0}^{C'-1}\left(B^v_j\bigcap_{i\in \llbracket k+2,\l\rrbracket} A^{'v}_{i}\right)\bigcup \left(\left\{\Gamma_{v,n}(2^{C'}x)-2^{C'}x\geq s_1((2\alpha)^{C'}x)^{2/3}\right\}\bigcap_{i\in \llbracket k+3,\l\rrbracket} A^{'v}_{i}\right).
\end{eqnarray*}
Now let $U_j=\{(2^jx,y):y\geq 2^jx+s_1((2\alpha)^jx)^{2/3}\}$ and $V_j=\{(2^{j+1}x,y): y\leq 2^{j+1}x+s_1((2\alpha)^{j+1}x)^{2/3}\}$ for $j\leq C'-1$. Then, for each $j\leq C'-1$,
\begin{eqnarray*}
&& \left(B^v_j\bigcap_{i\in \llbracket k+2,\l\rrbracket} A^{'v}_{i}\right)\\
&\subseteq & \left\{\sup_{u\in U_j,w\in V_j}(T_{v,u}+T_{u,w}-T_{v,w})\geq 0\right\}\bigcap \left\{\bigcup_{w\in V_j}\left\{\bigcap_{i\in \llbracket k+2,\l\rrbracket}A_i^{'w}\right\}\right\}
\end{eqnarray*}
The two events in the intersection are independent. It follows from the proof of Theorem \ref{t:carsestimate} that
\[\P\left\{\sup_{u\in U_j,w\in V_j}(T_{v,u}+T_{u,w}-T_{v,w})\geq 0\right\}\leq e^{-cs_1\alpha^{2j/3}}.\]
Also for all $j\leq C'-1$, using the induction hypothesis \eqref{e:ind1} with $r=k+1$, $s_1'=s$, $x'=M^{r+2}$, and $v'=(2^{j+1}x,2^{j+1}x+s_1((2\alpha)^{j+1}x)^{2/3})$,
the topmost vertex of $V_j$, and using polymer ordering Lemma \ref{l:porder},
\[\P\left\{\bigcup_{w\in V_j}\left\{\bigcap_{i\in \llbracket k+2,\l\rrbracket}A_i^{'w}\right\}\right\}\leq \P\left(\bigcap_{i\in \llbracket k+2,\l\rrbracket}A_i^{'v'}\right)\leq 2^{\l-k-1}e^{-cs(\l-k-1)}.\]
Also the fact that  $\alpha^{C'}<2^{C-C'}$ and the condition on $s_1$ imply that $s_1''((2\alpha)^{C'}x'')^{2/3}<s(M^{j+3})^{2/3}$, where $x''=2^{C'}x$ and $s_1''=s_1\alpha^{2C'/3}$ and $\{\Gamma(2^{C'}x)-2^{C'}x\geq s_1((2\alpha)^{C'}x)^{2/3}\}=E_{k+2,x'',s_1''}$. Hence, applying statement \eqref{e:ind1} of the induction hypothesis again,
\begin{eqnarray*}
&&\P\left(\left\{\Gamma_{v,n}(2^{C'}x)-2^{C'}x\geq s_1((2\alpha)^{C'}x)^{2/3}\right\}\bigcap_{i\in \llbracket k+3,\l\rrbracket} A^{'v}_{i}\right)\\
&=&\P\left(E_{k+2,x'',s_1''}\bigcap_{i\in \llbracket k+3,\l\rrbracket} A^{'v}_{i}\right)\\
&\leq & 2^{l-k-1}e^{-cs_1''-cs(\l-k-2)}=2^{l-k-1}e^{-cs_1\alpha^{2C'/3}-cs(\l-k-2)}.
\end{eqnarray*}

Hence, bringing all this together,
\[\P(B)\leq \sum_{j=1}^{C'-1}2^{\l-k-1}e^{-cs(\l-k-1)}e^{-cs_1\alpha^{2j/3}}+2^{l-k-1}e^{-cs_1\alpha^{2C'/3}-cs(\l-k-2)}.\]
Since $\alpha^{2C'/3}\geq 2$, this proves that
\[\P(B)\leq 2^{\l-k}e^{-cs_1-cs(\l-k-1)}.\]
This proves statement \eqref{e:ind1} of the induction hypothesis for $j=k$ and completes the induction and proves the claim.

 Hence, with $r=0$, $x=M$ and $s_1=s$, $z=0$, we get $E_{j+1,x,s_1}=A'_1$, and from the above claim,
\[\P\left(\bigcap_{i=1}^{\l}A'_i\right)\leq 2^{\l}e^{-c\l s}.\]
hence, by taking a union bound over all $2^{\l}$ terms,
\[\P(|\Gamma(M^i)-M^i|\geq s(M^i)^{2/3} \mbox{ for all } i \in [\l])\leq 2^{2\l}e^{-c\l s}.\]
For all $s\geq s_0$, such that $\frac{c}{2}s>\log 4$, one has the result.
\end{proof}

Note that the arguments used in proving Theorem \ref{t:carsestimate} and Theorem \ref{t:carsbetter} would still go through if we considered the transversal fluctuation of the geodesic between two points such that
the line segment joining them has slope $m$ bounded away from $0$ and $\infty$. We state this without proof in the following corollary.

\begin{corollary}
Let $\psi>1$ and $m\in [\frac{1}{\psi},\psi]$ be fixed. Let $\Gamma$ be the geodesic from $(0,0)$ to $(n,mn)$. Let $\mathcal{S}$ be the line segment joining $(0,0)$ to $(n,mn)$. For $\l\in \Z$, let $\mathcal{S}(\l)$ be such that $(\l,\mathcal{S}(\l))\in \mathcal{S}$.
\begin{itemize}
\item[(a)] \label{carsestimate2}
 Then there exist positive constants $n_0, \l_0, s_0, c$ depending only on $\psi$, such that for all $n\geq n_0, s\geq s_0, \l\geq \l_0$,
\[P[|\Gamma(\l)-\mathcal{S}(\l)|\geq s\l^{2/3}] \leq e^{-cs}.\]
\item[(b)]\label{c:carsbetter2}
 For $k, M\in \Z$ fixed, and any vertex $v\in \Z^2$, let $A'_{i}$ denote the event
that $|\Gamma(M^{i}k)-\mathcal{S}(M^{i}k)| \geq s (M^{i}k)^{2/3}$. Then there exist positive constants $s_0, M_0, c,c'$ depending only on $\psi$, such that for all $s>s_0$, $M>M_0$ and $\ell,k\in \Z$ we have
$$\limsup_{n\to \infty} \P\left(\sum_{i=1}^{\ell} \ind_{A'_{i}}\geq 2e^{-cs}\ell\right) \leq e^{-c'\ell}.$$
\end{itemize}
\end{corollary}

\section{Coalescence of Finite Geodesics}
\label{s:coales}
In this section we prove Theorem \ref{t:coal} following the strategy outlined in Section \ref{s:outline}. We prove Theorem \ref{t:finite1} modulo the key Proposition \ref{l:meetsubinter} stated below, and use it to complete the proof of Theorem \ref{t:coal}.


Recall the set-up of Theorem \ref{t:finite1}. Let $L$ be a large fixed number. For $r\in \N$ consider the points $a_{r}=(r,r+Lr^{2/3})$ and $b_{r}=(r,r-Lr^{2/3})$ (assume without loss of generality that $r^{2/3}$ is an integer). For some large $M\in \N$ let $\Gamma_1$ denote the geodesic from $a_{r}$ to $a_{Mr}$, and $\Gamma_2$ denote the geodesic from $b_r$ to $b_{Mr}$. Let ${\rm Coal}_{r,M}$ denote the event that $\Gamma_1$ and $\Gamma_2$ share a common vertex. We want to show that the probability of this event is bounded below by $1/2$, for some suitably large $M$, uniformly in all large $r$. We first show that the probability is bounded below by an arbitrary small constant independent of $r$ and $M$.

%


\begin{proposition}\label{l:meetsubinter}
 Fix $z,r\in \N$ and $0\leq u_0\leq \log \log z$ and $0\leq v_0 \leq \log \log z^{2}$. Set $a_1=(zr,zr+u_0z^{2/3}r^{2/3}), b_1=(zr,zr-u_0z^{2/3}\l^{2/3}), a_2=(z^2r,z^2r+v_0(z^2)^{2/3}r^{2/3}),b_2=(z^2r,z^2r-v_0(z^2)^{2/3}r^{2/3})$. Let $\Gamma_0$ be the geodesic from $a_1$ to $a_2$ and $\Gamma_0'$ be the geodesic from $b_1$ to $b_2$. Let $F$ be the event that $\Gamma_0$ and $\Gamma_0'$ meet one another. Then there exists an absolute positive constant $\alpha$ not depending on $z,r$, such that
\[\P(F)\geq \alpha>0.\]
\end{proposition}

Proof of Proposition \ref{l:meetsubinter} is rather elaborate and the next two sections are devoted to it. Roughly the idea is as follows. We show that at some horizontal length scale $D$ (where $r\ll D \ll Mr$) with a probability bounded away from zero there exists a barrier of width $O(D)$ and height $O(D^{2/3})$ just above the geodesic $\Gamma_{2}$, such that any path passing through the barrier is penalised a lot. Using Theorem \ref{t:carsestimate} and the FKG inequality we show that in presence of such a barrier and in environment that is typical otherwise, $\Gamma_{2}$ will merge with $\Gamma_{1}$ before crossing the barrier region. The construction of the barrier here is similar to one present in \cite{BSS14}, and we shall quote many of the probabilistic estimates in that paper throughout our proof. 

We  first show how we can conclude Theorem \ref{t:finite1} from Proposition \ref{l:meetsubinter}.


\begin{proof}[Proof of Theorem \ref{t:finite1}] Fix $r$. By translation invariance, it is enough to look at the geodesic from $e_0=(0,Lr^{2/3})$ to $e_1=((M-1)r,(M-1)r+L(Mr)^{2/3})$ and the geodesic from $g_0=(0,-Lr^{2/3})$ to $g_1=((M-1)r,(M-1)r-L(Mr)^{2/3})$. Let $\Gamma=\Gamma_{e_0,e_1}$ and $\Gamma'=\Gamma_{g_0,g_1}$. Set
\[p_i=2^{2^i}r, \mbox{ for } i=\frac{N}{2},\frac{N}{2}+1,\ldots,N,\]
such that $2^{2^N}=M-1$, i.e., $N=c\log \log (M-1)$ for some absolute constant $c$. Also set $a_i=(p_i,p_i+p_i^{2/3}i)$ and $b_i=(p_i,p_i-p_i^{2/3}i)$, for $i=\frac{N}{2},\frac{N}{2}+1,\ldots,N$. Applying Lemma \ref{carsestimate2}, it follows that $\Gamma$ and $\Gamma'$ pass between $a_i$ and $b_i$ for all $i$ with probability at least
\[1-2\sum_{i=\frac{N}{2}}^{N} e^{-c
\frac{i}{2}}\geq 1-C'e^{-c'N}.\]
From Proposition \ref{l:meetsubinter}, with $z=2^{2^{i}}$, it is easy to see that $\Gamma_{a_i,a_{i+1}}$ and $\Gamma_{b_i,b_{i+1}}$ meet with a positive probability $\alpha$ not depending on $i,r$. Hence, due to polymer ordering (Lemma \ref{l:porder}) and independence in each block, $\Gamma$ and $\Gamma'$ meet between $0$ and $(M-1)r$ with probability at least
\[1-(1-\alpha)^{N/2}-C'e^{-c'N} \geq 1-\frac{C'}{(\log M)^{c'}},\]
for some positive constants $C',c'$ that depend only on $L$ but not on $r$ or $M$. This proves the theorem by choosing $M$ sufficiently large.
\end{proof}

We now complete the proof of Theorem \ref{t:coal}. The proof is similar to the proof of Theorem \ref{t:finite1}, except that we use Theorem \ref{t:carsbetter} instead of Theorem \ref{t:carsestimate} and a union bound.


\begin{figure}[h]
\centering
\includegraphics[width=0.5\textwidth]{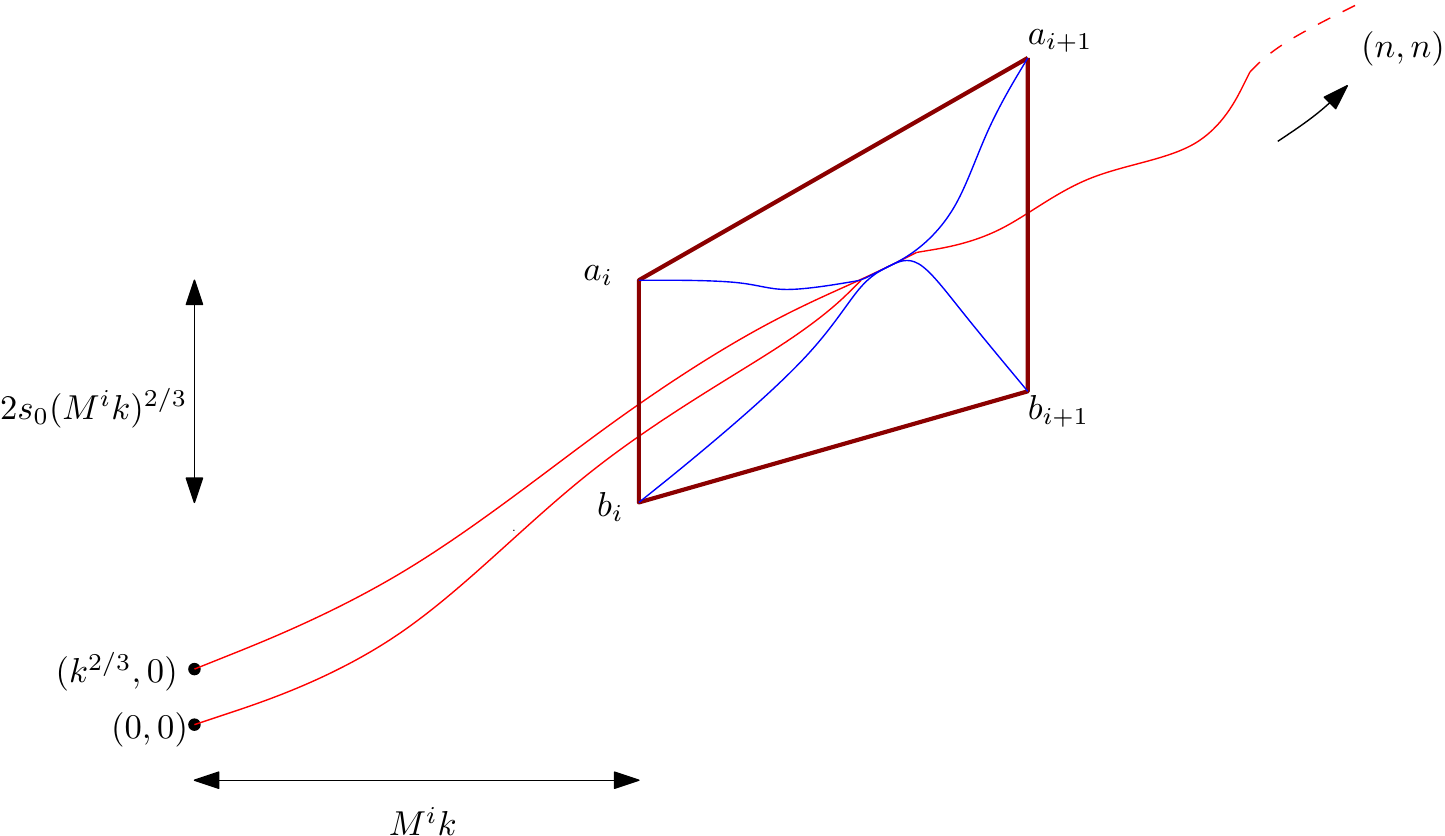}
\caption{Merging of paths as in the proof of Theorem \ref{t:coal}. Using Theorem \ref{t:carsbetter} it follows that the geodesics $\Gamma_{1}$ and $\Gamma_2$ are very likely to pass between the points $a_{i}$ and $b_{i}$ for all $i$ sufficiently large. Using Theorem \ref{t:finite1} we show that with a positive probability they merge in one of those intervals. Proof of Theorem \ref{t:finite1} is similar, except there we choose intervals growing doubly exponentially and the distance between the points $a_i$ and $b_i$ much larger. The argument proceeds by using Proposition \ref{l:meetsubinter} instead of Theorem \ref{t:finite1} and Theorem \ref{t:carsestimate} together with a union bound replacing Theorem \ref{t:carsbetter}.}
\label{f:merging}
\end{figure}

\begin{proof}[Proof of Theorem \ref{t:coal}]
Let $M>M_0,s_0$ as in Theorem \ref{t:carsbetter} and let $e^{-cs_{0}}=\epsilon< \frac{1}{8}$. Let $C_{i}$ denote the event that $|\Gamma_{v_1,\mathbf{n}}(M^{i}k)-M^{i}k|\leq s_0(M^ik)^{2/3}$ and $|\Gamma_{v_2,\mathbf{n}}(M^{i}k)-M^{i}k|\leq s_0(M^ik)^{2/3}$. Let $a_{i}=(M^{i}k, M^{i}k+s_0(M^ik)^{2/3})$ and $b_i=(M^{i}k, M^{i}k-s_0(M^ik)^{2/3})$. Let $D_i$ denote the event that the geodesic from $a_{i}$ to $a_{i+1}$ meets the geodesic from $b_i$ to $b_{i+1}$. By Theorem \ref{t:finite1} and choosing $r=M^ik$, it follows that for $M$ sufficiently large we have $\P(D_i)\geq 1/2$ for all $i$. Since $D_i$ are independent it follows that for all $\ell$
$$ \P\left(\sum_{i=1}^{\ell} \ind_{D_{i}}\leq \frac{\ell}{4}\right)\leq e^{-c'\ell}$$
for some $c'>0$. Using this together with Theorem \ref{t:carsbetter} we get

$$ \P\left(\sum_{i=1}^{\log R (\log M)^{-1}} \ind_{C_{i}\cap D_{i}}=0\right)\leq R^{-c},$$
for some absolute positive constant $c$. This proves Theorem \ref{t:coal}.
\end{proof}

The idea of the proofs of Theorem \ref{t:finite1} and Theorem \ref{t:coal} is illustrated in Figure \ref{f:merging}. It will be clear from our proof that Theorem \ref{t:finite1} works for geodesics in any fixed direction bounded away from the co-ordinate axes directions. This, together with Corollary \ref{c:carsbetter2} implies the following corollary, which we state without proof.

\begin{corollary}
\label{c:pathmeetslope}
Let $L,m_0$ be two fixed positive constants and let $n\gg k$.  Let $\Gamma$ be the geodesic from $(0,Lk^{2/3})$ to $(n,m_0n+Ln^{2/3})$ and $\Gamma'$ be the geodesic from $(0,-Lk^{2/3})$ to $(n,m_0n-Ln^{2/3})$ in the exponential LPP model. Let $v_*=(v_{*,1},v_{*,2})$ be a leftmost common point between $\Gamma$ and $\Gamma'$. Then,
\[\P(v_{*,1}>Rk)\leq CR^{-c},\]
for some positive constants $C,c$ that depend on $m_0,L$ but not on $k,n$.
\end{corollary}

\textbf{To complete the proof of Theorem \ref{t:coal}, it only remains to prove Proposition \ref{l:meetsubinter}. This is done over the next two sections. Observe that by scaling of the process, it suffices to prove Proposition \ref{l:meetsubinter} for $r=1$. For the next two sections we shall always be in this setting even though we might not explicitly mention it every time.}

\section{Favourable events}
\label{s:events}
First we need to define a set of events that will be key to our proof. Some of these are similar to the events in Section $3$ of \cite{BSS14}. We need to introduce some more notations before we can define these events. 

\subsection*{More notations}
Fix $z\in\N$ and consider the set up as in Proposition \ref{l:meetsubinter} with $r=1$. Define $x=z^{3/2}$. In this section, all the events are constructed for this fixed $z$. As there is no scope for confusion, we suppress the dependence of the events on $z$. Recall that $\Gamma_0$ is the geodesic from $a_1$ to $a_2$ and $\Gamma_0'$ be the geodesic from $b_1$ to $b_2$. 

Let $\mathcal{P}(w,\l,h,s)$ denote the parallelogram whose leftmost endpoint is $(w,w-h)$, and two sides are parallel to the diagonal and $y$-axis and are of length $\sqrt{2}\l$ and $h+s$ respectively, i.e., whose endpoints are $(w,w-h),(w+\l,w+\l-h),(w,w+s),(w+\l,w+\l+s)$. Construct the barrier $B$ at $x$ of width $x/10$ and height $(4M+S)x^{2/3}$ as follows,
\[B=\mathcal{P}(x,x/10,2Mx^{2/3},(2M+S)x^{2/3}).\]
 Let us denote the left wall of the barrier as $L_1$ and the right wall $L_2$.
Also let
\[\mathcal{Z}=\{(u,v) \in \R^2: x\leq u\leq x+\frac{x}{10}\}\]
be the region bounded by the vertical lines at $x$ and $x+\frac{x}{10}$. Define $x'=(2+\frac{1}{10})x$ and let $L_3$ denote the line segment joining $(x',x'-2M(x')^{2/3})$ and $(x',x'+(2M+S)(x')^{2/3})$. See Figure \ref{f:bf} for an illustration of the above definitions.

\subsection*{Choice of Parameters}
The construction of the favourable events will depend on a number of parameters. In the definitions that follow $H,M,S$ will denote large positive constants to be chosen appropriately later (not depending on $z$). The dependence among these constants are as follows.
\begin{itemize}
\item[1.] $H$ will denote a large absolute constant.
\item[2.] $M$ will denote a large absolute constant.\footnote{Note that the parameter $M$ here and in subsequent sections is in no way related to the constant $M$ used in the earlier Sections \ref{s:cars} and \ref{s:coales}.}
\item[3.] $S$ is chosen sufficiently large depending on $M$ and $H$.
\end{itemize}

\begin{figure}[htb!]
\centering
\includegraphics[width=0.3\textwidth]{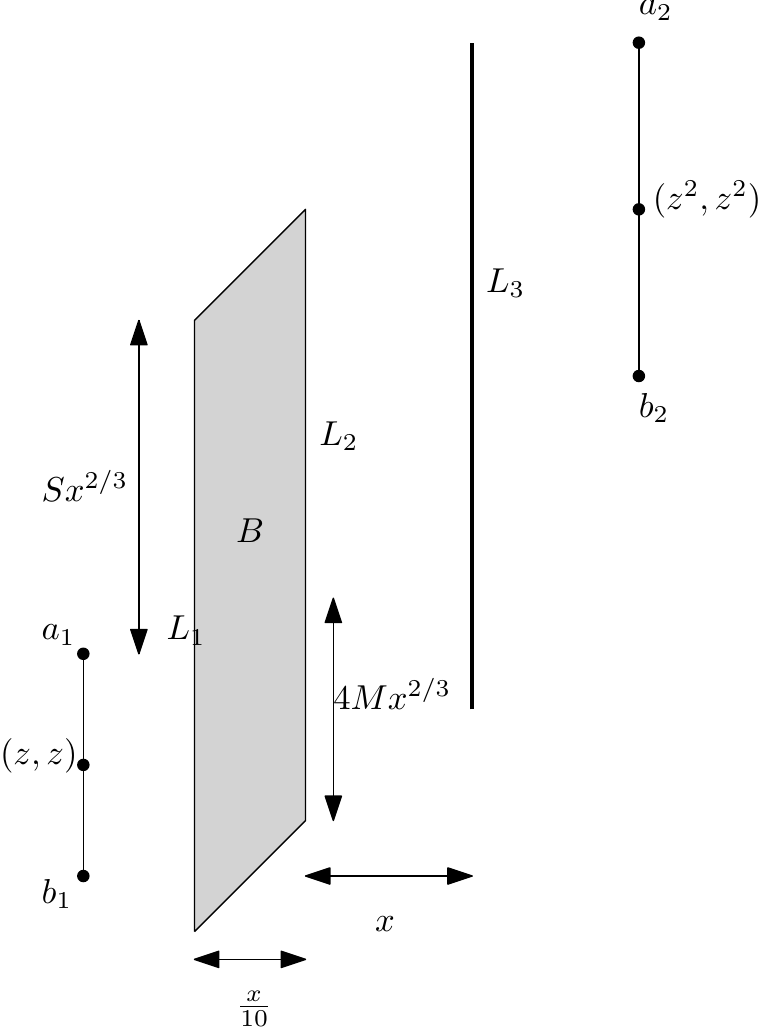}
\caption{The basic elements of our construction, the barrier $B$, and the line segments $L_1,L_2$ and $L_3$. Notice that $x=z^{3/2}$ so the barrier is much closer to the left boundary than to the right one.}
\label{f:bf}
\end{figure}



With this preparation we can now define the favourable events, which are divided into four types. The first three are typical in the sense that they hold with probability close to 1 (for the appropriate choice of parameters) but the final one only occurs with probability bounded away from 0.

\subsection{Wing condition}

Fix some large absolute constant $H$ to be chosen appropriately later. We say $G$ holds if the following two conditions hold:

\begin{enumerate}
\item[(i)] For all $u\in L_1$,  the left wall of the barrier,
\[|\widetilde{T}_{a_1,u}|\leq H\sqrt{S}x^{1/3}.\]
\item[(ii)] For all $u'\in L_2$, and $v\in L_3$,
\[|\widetilde{T}_{u',v}|\leq H\sqrt{S}x^{1/3}.\]
\end{enumerate}

The point of this condition is to ensure that the passage times to the left and the right of the barrier region behave typically. It follows from Lemma $7.3$ in \cite{BSS14} that the event $G$ holds with high probability, i.e., by choosing $H$ a large absolute constant,
\[\P(G)\geq \frac{99}{100}.\]
Observe that $G$ depends only on the configuration in $\cZ^c$.

\subsection{Typical path}
 Let $m=\frac{11}{10}$, so that the barrier ends at the vertical line $y=mx$, and recall $x'=(m+1)x$. Let $\gamma$ be an increasing path from $b_1$ to $b_2$. Define,
\[\widetilde{\l(\gamma[x,mx])}:=\l(\gamma[x,mx])-\E(T_{(x,\gamma(x)),(mx,\gamma(mx))}),\]
and
\[\widehat{\l(\gamma[x,mx])}:=\l(\gamma[x,mx])-2d((x,\gamma(x)),(mx,\gamma(mx))),\]
where $\gamma[x,mx]$ is the part of $\gamma$ between $x$ and $mx$ and $\l(\gamma)$ is the weight of $\gamma$. We say $\gamma$ is \textbf{typical} at location $x$, if the weight of $\gamma[x,mx]$ behave typically, and a series of geometric conditions hold ensuring $\gamma$ passes through the bottom part of $B$ and fluctuates at the typical transversal fluctuation scale of geodesics in the region between $L_1$ and $L_3$. More concretely, we ask for the following conditions:  

\textbf{Weight conditions:}
\begin{equation}\label{e:ltilde}
\left|\widetilde{\l(\gamma[x,mx])}\right|\leq H\sqrt{M}x^{1/3}.
\end{equation}
\begin{equation}\label{e:lhat}
\left|\widehat{\l(\gamma[x,mx])}\right|\leq H\sqrt{M}x^{1/3}.
\end{equation}

\textbf{Geometric Conditions:}
\begin{equation}\label{gammadeviation}
x+Mx^{2/3}\geq  \gamma(x) \geq x-Mx^{2/3},
\end{equation}
\begin{equation}\label{gammadeviation2}
mx+M(mx)^{2/3}\geq  \gamma(mx) \geq mx-M(mx)^{2/3},
\end{equation}
\begin{equation}\label{gammadeviation3}
x'+M(x')^{2/3}\geq  \gamma(x') \geq x'-M(x')^{2/3}.
\end{equation}

\begin{equation}\label{staysinbarrier}
\{(t,\gamma(t)):x\leq t\leq mx\}\subseteq \mathcal{P}(x,x/10,2Mx^{2/3},2Mx^{2/3})=:B_0.
\end{equation}

Note that conditions \eqref{e:ltilde} and \eqref{e:lhat} involve $\Pi_{\gamma[x,mx]}$, the configuration on $\gamma[x,mx]$, and the rest are conditions on the geometric properties of the path $\gamma$. See Figure \ref{f:typical} for an illustration.
We shall show later that geodesics are typical with high probability.

\begin{figure}[htb!]
\centering
\includegraphics[width=0.3\textwidth]{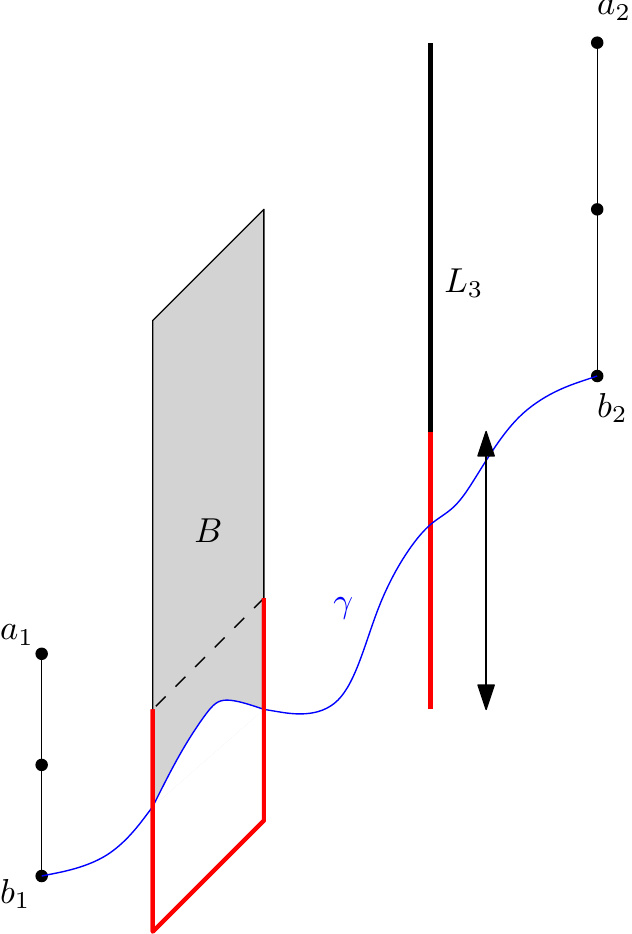}
\caption{Typical paths and barrier condition: a path $\gamma$ from $b_1$ to $b_2$ is typical if it has typical transversal fluctuations while crossing the barrier region and whose length restricted to the barrier region is typical. The barrier condition $R_{\gamma}$ asserts that the region in the barrier above $\gamma$ is really bad in the sense that any path crossing the barrier from left to right above $\gamma$ and is disjoint with $\gamma$  is much smaller than typical length.}
\label{f:typical}
\end{figure}

\subsection{Path condition}
Fix any increasing path $\gamma$ from $b_1$ to $b_2$. Our next favourable event asks that paths from $a_1$ to $a_2$ that have atypical transversal fluctuations in order to avoid crossing the barrier region will be not competitive with paths that simply cross the barrier region coinciding with $\gamma$. 

Let $T_{\gamma,x,mx}$ denote the weight of the best path from $a_1$ to $a_2$ that coincides with $\gamma[x,mx]$ between $x$ and $mx$. Let $F^1_{\gamma}$ be the weight of the best path from $a_1$ to $a_2$ that is more than a distance of $(2M+S)x^{2/3}$ above the diagonal at $x$,(i.e., at $x$ is above the left boundary of $B$) and stays above $\gamma[x,mx]$ in $[x,mx]$. Let
\[A^1_\gamma=\{F^1_{\gamma}<T_{\gamma,x,mx}-\sqrt{S}x^{1/3}\}.\]
Similarly, let $F^2_{\gamma}$ be the weight of the best path from $a_1$ to $a_2$ that is more than a distance of $(2M+S)(mx)^{2/3}$ above the diagonal at $mx$, (i.e., passes above the right boundary of $B$) and stays above $\gamma[x,mx]$ in $[x,mx]$, and define
\[A^2_\gamma=\{F^2_{\gamma}<T_{\gamma,x,mx}-\sqrt{S}x^{1/3}\}.\]
Also let $F^3_{\gamma}$ be the weight of the best path from $a_1$ to $a_2$ that is more than a distance of $(2M+S)(x')^{2/3}$ above the diagonal at $x'$, (i.e., passes above $L_3$) and stays above $\gamma[x,mx]$ between $[x,mx]$, and define
\[A^3_\gamma=\{F^3_{\gamma}<T_{\gamma,x,mx}-\sqrt{S}x^{1/3}\}.
\]
Define
\[A_\gamma=A^1_\gamma\cap A^2_\gamma\cap A^3_\gamma.\]

Observe that the event $A_\gamma$ is decreasing on the configuration of $\cZ\setminus \{\gamma\}$ conditioned on the remaining configuration. See Figure \ref{f:path}.

\begin{figure}[htb!]
\centering
\includegraphics[width=0.3\textwidth]{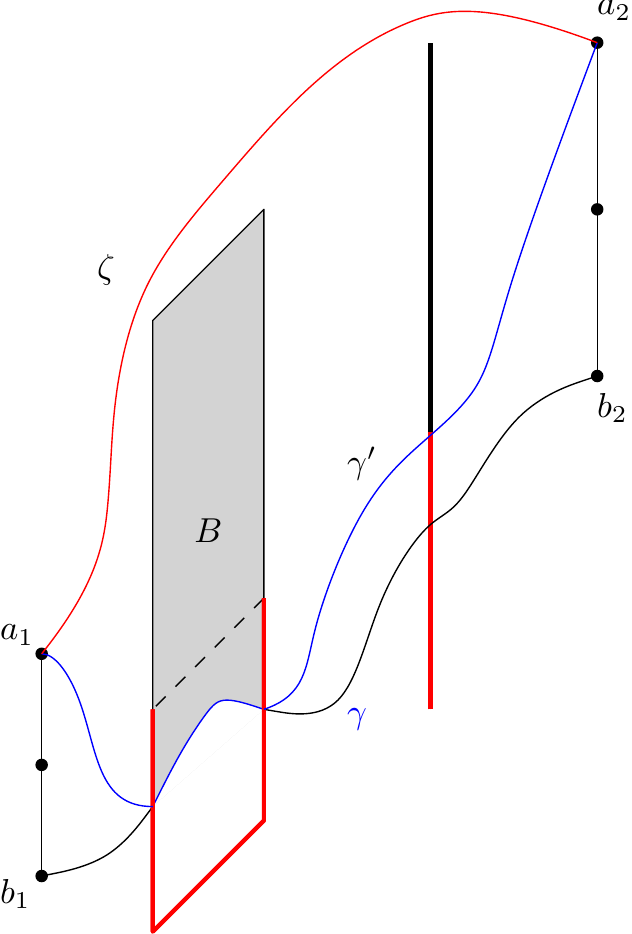}
\caption{Path condition $A_{\gamma}$: it asserts that a path $\zeta$ which passes above either wall of the barrier or the line segment $L_3$ must be much smaller than the path $\gamma'$ which is the longest path from $a_1$ to $a_2$ that agrees with $\gamma$ along the barrier.}
\label{f:path}
\end{figure}

Recall that $\Gamma_0'$ is the geodesic from $b_1$ to $b_2$. Define $A_{\Gamma_0'}$ similarly with $\gamma$ replaced by $\Gamma_0'$. We shall show later that $A_{\Gamma_0'}$ is a high probability event.

\subsection{Barrier condition}
So far all the events that we have described are events that typically hold. Our final favourable event is one that ensures any path crossing the barrier disjointly with $\Gamma'_0$ will be penalised a lot. This is not a typical event but one that only holds with constant probability (independent of $z$).
Fix an increasing path $\gamma$ from $b_1$ to $b_2$ satisfying the geometric conditions \eqref{gammadeviation}, \eqref{gammadeviation2} and \eqref{staysinbarrier}. We say the barrier condition $R_{\gamma}$ holds if:

Any path from the left to the right wall of the barrier $B$ that avoids $\gamma$ and crosses the barrier is much shorter. That is, for all $u\in L_1$ and $u'\in L_2$, such that $(x,\gamma(x))\preceq u$ and $(mx,\gamma(mx))\preceq u'$,
\[\widetilde{T}_{u,u'}^{\gamma} \leq -S^4x^{1/3}.\]
It follows from Lemma $8.3$ of \cite{BSS14}, that there exists a constant $\beta>0$ not depending on $x$ and $\gamma$, (depending on $M,S$), such that
\[\P(R_{\gamma})\geq \beta>0.\]
Observe that $R_{\gamma}$ depends only on the configuration in $\cZ\setminus \{\gamma\}$ and is decreasing in $\cZ\setminus\{\gamma\}$.

\section{Forcing Geodesics to Merge Using Favourable Events}
\label{s:const}

In order to show that a collection of geodesics coalesce, we shall need the following lemmas about the events defined in the previous subsection. Recall that $z\in \N$ is fixed, the set up is as given in Lemma \ref{l:meetsubinter} with $r=1$ and $x=z^{3/2}$. Now we consider the events described in the previous section. Then we have the following lemmas. The first lemma says that geodesics are \emph{typical} with high probability.

\begin{lemma}\label{typicalpath}Let $\Gamma_0'$ be the geodesic from $b_1$ to $b_2$. Then $\Gamma_0'$ is \textbf{typical} with probability at least $\frac{99}{100}$.
\end{lemma}
\begin{proof}
We first show that conditions \eqref{gammadeviation}, \eqref{gammadeviation2} and \eqref{gammadeviation3} each occur with probability at least $\frac{999}{1000}$. Enough to show one of them, the others are similar. Let $\mathcal{S}$ be the line segment joining the two points $b_1$ and $b_2$. Then at the point $(x,x)$ on the diagonal, $x\geq \mathcal{S}(x)\geq x-\frac{1}{4}x^{2/3}$ (where $\mathcal{S}(x)$ is such that $(x,\mathcal{S}(x))\in \mathcal{S}$). Hence, by Corollary \ref{carsestimate2},  \eqref{gammadeviation} occurs with probability at least $\frac{999}{1000}$ by choosing $M$ large.

Next we show that conditions \eqref{e:ltilde} and \eqref{e:lhat} hold with high probability. Let $S_1$ be the line segment joining $(x,x-Mx^{2/3})$ and $(x,x+Mx^{2/3})$, and $S_2$ be the line segment joining $(mx,mx-M(mx)^{2/3})$ and $(mx,mx+M(mx)^{2/3})$. Let $A$ denote the event that
\[\sup_{u\in S_1,u'\in S_2}|\widetilde{T}_{u,u'}|\leq \frac{H}{2}\sqrt{M}x^{1/3}.\]
That $\P(A)\geq \frac{999}{1000}$ follows from Corollary $10.4$ and Corollary $10.7$ of \cite{BSS14} and hence condition \eqref{e:ltilde} holds with high probability. Let $A'$ be the event that $|\Gamma_0'(mx)-\Gamma_0'(x)|\leq \frac{\sqrt{H}}{2}M^{1/6}x^{2/3}$. Breaking $S_1$ into subintervals of length $\frac{\sqrt{H}}{2}M^{1/6}x^{2/3}$, using Corollary \ref{carsestimate2} for the geodesics starting from each of the endpoints of the subintervals to $b_2$, and polymer ordering (Lemma \ref{l:porder}) and union bound gives that $\P(A')\geq \frac{999}{1000}$. Observe that for any geodesic $\gamma$ from $u$ to $u'$ with $u\in S_1,u'\in S_2$ , such that $|\widetilde{\l(\gamma})|\leq \frac{H\sqrt{M}x^{1/3}}{2}$ and $|\gamma(u')-\gamma(u)|\leq \frac{\sqrt{H}}{2}M^{1/6}x^{2/3}$, one has $|\widehat{\l(\gamma})|\leq H\sqrt{M}x^{1/3}$. This, together with condition \eqref{gammadeviation} gives that condition \eqref{e:lhat} occurs with probability at least $\frac{997}{100}$.

The only thing left to show is that condition \eqref{staysinbarrier} occurs with probability at least $\frac{999}{1000}$. Together with polymer ordering (Lemma \ref{l:porder}) and conditions \eqref{gammadeviation} and \eqref{gammadeviation2}, it is enough to show that the geodesic joining $(x,x-Mx^{2/3})$ and $(mx,mx-M(mx)^{2/3})$, and that joining $(x,x+Mx^{2/3})$ and $(mx,mx+M(mx)^{2/3})$ have \emph{transversal fluctuation} at most $\frac{M}{2}x^{2/3}$ with high probability. This follows from Theorem 11.1 of \cite{BSS14} by choosing $M$ a large constant.
\end{proof}

The proof of the next lemma is similar to that of Theorem \ref{t:carsestimate}.

\begin{lemma}\label{l:Agamma}
For any fixed increasing path $\gamma$ from $b_1$ to $b_2$ that satisfies the geometric conditions
\eqref{gammadeviation}, \eqref{gammadeviation2}, \eqref{gammadeviation3} and \eqref{staysinbarrier}, $\P(A_\gamma|\gamma \mbox{ is typical })\geq \frac{97}{100}$.
\end{lemma}

\begin{proof}
 Observe that for an increasing path $\gamma$ satisfying the geometric conditions in \eqref{gammadeviation}, \eqref{gammadeviation2}, \eqref{gammadeviation3}, \eqref{staysinbarrier}, the condition that it is typical depends only on the configuration $\Pi_{\{\gamma[x,mx]\}}$. We show that $\P(A^1_\gamma|\gamma \mbox{ is typical})\geq \frac{99}{100}$. The bounds for $A^2_\gamma$ and $A^3_\gamma$ follow similarly and that for $A_\gamma$ follows by taking a union bound. The proof of this is similar to what we did for the proof of Theorem \ref{t:carsestimate}.

See Figure \ref{f:path}. Let $\zeta$ be the best path joining $a_1$ and $a_2$ that is more than a distance of $(2M+S)x^{2/3}$ above the diagonal at $x$ and is above $\gamma[x,mx]$ in $[x,mx]$. Let $\mathcal{W}$ be the straight line segment joining $a_1$ and $a_2$. Choose $\alpha=2^{\frac{1}{6}}$.
For $j\geq 0$, let $B'_j$ denote the event that $\zeta(2^jx)-\mathcal{W}(2^jx)\geq (2M+S)((2\alpha)^jx)^{2/3}$ and $\zeta(2^{j+1}x)-\mathcal{W}(2^{j+1}x)\leq (2M+S)((2\alpha)^{j+1}x)^{2/3}$. It is enough to show that on each of these $B'_j$, the weight of the union of the maximal path joining $a_1$ to $(x,\gamma(x))$, $\gamma[x,mx]$, and the maximal path joining $(mx,\gamma(mx))$ and $(2^{j+1}x,\zeta(2^{j+1}x))$ is larger than the sum of the length of $\zeta[z,2^{j+1}x]$ and $\sqrt{S}x^{1/3}$ with sufficiently high probability.


As before, define $U_r$ as the line segment joining $(2^jx,\mathcal{W}(2^jx)+(2M+S+r)((2\alpha)^jx)^{2/3})$ and $(2^{j}x,\mathcal{W}(2^jx)+(2M+S+r+1)((2\alpha)^jx)^{2/3})$ and $V_r$ as the line segment joining $(2^{j+1}x,\mathcal{W}(2^{j+1}x)+(2M+S-r)((2\alpha)^{j+1}x)^{2/3})$ and $(2^{j+1}x,\mathcal{W}(2^{j+1}x)+(2M+S-r+1)((2\alpha)^{j+1}x)^{2/3})$. Note that $r$ here is used as an index variable and, in particular is not related to the same symbol used in statement of Proposition \ref{l:meetsubinter}. Also recall that $S_1$ was the line segment joining $(x,x-Mx^{2/3})$ and $(x,x+Mx^{2/3})$ and $S_2$ was the line segment joining $(mx,mx-M(mx)^{2/3})$ and $(mx,mx+M(mx)^{2/3})$. Define
$$D_{x,r,r', j}=\sup_{u\in U_r,v\in V_{r'},w_1\in S_1,w_2\in S_2} \biggl(T_{a_1,(2^jx,u)}+T_{(2^jx,u),(2^{j+1}x,v)}-T_{a_1,(x,w_1)}-T_{(mx,w_2),(2^{j+1}x,v)}\biggr)$$
and set 
$$C_{j,r,r'}=\left\{ D_{x,r,r', j} \geq  \l(\gamma[x,mx])-\sqrt{S}x^{1/3}\right\}.$$
Recall that $\ell(\gamma)$ denotes the length of the path $\gamma$. 
%
Computing expectations, it is easy to see that for some constant $c_1$ not depending on $x,S,r,r',j$ (depending on $M$), for all $u\in U_r, v\in V_{r'}$, (observe that $\gamma(x)\in S_1, \gamma(mx)\in S_2$),
\begin{eqnarray*}
\E(T_{a_1,(2^jx,u)})+\E(T_{(2^jx,u),(2^{j+1}x,v)})\leq \E(T_{a_1,(x,\gamma(x))})+\E(T_{(x,\gamma(x)),(mx,\gamma(mx))})+\E(T_{(mx,\gamma(mx)),(2^{j+1}x,v)})&&\\
-c_1(S+r+r')\alpha^{\frac{2j}{3}}(2^jx)^{1/3}.
\end{eqnarray*}
Using the moderate deviation estimates for supremum and infimum of the lengths of a collection of geodesics given in Proposition $10.1$ and Proposition $10.5$ of \cite{BSS14} and the fact that $\gamma$ is a typical path (hence condition \eqref{e:ltilde} holds), this implies, choosing $S$ large enough compared to $M$,
\[\P(C_{j,r,r'})\leq Ce^{-c(\frac{S}{2}+r+r')\alpha^{\frac{2j}{3}}}.\]
Summing over $r,r',j$, and choosing $S$ large enough, gives the result.
\end{proof}

We point out that the argument in the proof of Lemma \ref{l:Agamma} is useful in other contexts also. We already know from Theorem \ref{t:carsestimate} that the transversal fluctuation of a geodesic from $\mathbf{0}$ to $\mathbf{n}$ at $r\ll n$ is  $O(r^{2/3})$. The argument above shows the following stronger fact: any path having transversal fluctuation $\gg r^{2/3}$ at scale $r$ will typically be much shorter than the geodesic. See \cite{BG18} for an application of such a result. 

\medskip

The next lemma states that $A_{\Gamma_0'}$ is a high probability event.
\begin{lemma}\label{l:AGam}
Let $\Gamma_0'$ be the geodesic from $b_1$ to $b_2$. Then $\P(A_{\Gamma_0'})\geq \frac{9}{10}$.
\end{lemma}
\begin{proof}
Let $U_1$ be the line segment joining $(x,x-Mx^{2/3})$ and $(x,x+Mx^{2/3})$, $U_2$ be the line segment joining $(mx,mx-M(mx)^{2/3})$ and $(mx,mx+M(mx)^{2/3})$ and $U_3$ be the line segment joining $(x',x'-2M(x')^{2/3})$ and $(x',x'+2M(x')^{2/3})$.

Let $D_1$ be the event that the geodesic from $(mx,mx+M(mx)^{2/3})$ to $a_2$ and the geodesic from $(mx,mx-M(mx)^{2/3})$ to $a_2$ are within $2M(x')^{2/3}$ distance from the diagonal at $x'$. By Corollary \ref{carsestimate2} it follows that $\P(D_1)\geq \frac{999}{1000}$. This, together with polymer ordering, ensures that all geodesics from some point in $U_2$ to $a_2$ pass through $U_3$ with probability at least $\frac{999}{1000}$.

Let $D_2$ be the event such that the followings happen:

For all $u_1\in U_1$,
\[|\widetilde{T}_{a_1,u_1}|\leq H\sqrt{M}x^{1/3}.\]

For all $u_2\in U_2$ and $u_3\in U_3$,
\[|\widetilde{T}_{u_2,u_3}|\leq H\sqrt{M}x^{1/3}.\]

By Lemma $7.3$ of \cite{BSS14}, it follows that $\P(D_2)\geq \frac{999}{1000}$.

Observe that for any $v_1,v_2\in U_1$ and $w_1,w_2\in U_2$, $z\in U_3$
\begin{equation}\label{e:expect}
|\E(T_{a_1,v_1})+\E(T_{v_1,w_1})+\E(T_{w_1,z})-
(\E(T_{a_1,v_2})+\E(T_{v_2,w_2})+\E(T_{w_2,z}))|\leq cM^2x^{1/3},
\end{equation}
where $c$ is some absolute positive constant.

Let $\mathcal{Q}$ be the set of increasing paths from $b_1$ to $b_2$ that satisfy the geometric conditions \eqref{gammadeviation}, \eqref{gammadeviation2}, \eqref{gammadeviation3}, \eqref{staysinbarrier}. Fix two paths $\gamma_1,\gamma_2\in \mathcal{Q}$. Let $\Gamma^{(1)}$ and $\Gamma^{(2)}$ be the two best paths from $a_1$ to $a_2$ that coincide with $\gamma_1[x,mx]$ and $\gamma_2[x,mx]$ between $x$ and $mx$ respectively. Let $v_1=(x,\gamma_1(x)), v_2=(x,\gamma_2(x)), w_1=(mx,\gamma_1(mx)), w_2=(mx,\gamma_2(mx)), z_1=(x',\Gamma^{(1)}(x')), z_2=(x',\Gamma^{(2)}(x'))$.  Since $T_{w_2,a_2}\geq T_{w_2,z_1}+T_{z_1,a_2}$ and $T_{w_1,a_2}\geq T_{w_1,z_2}+T_{z_2,a_2}$, it is easy to see that, on $D_1$,
\[|T_{\gamma_1,x,mx}-T_{\gamma_2,x,mx}|\leq \sup_{z\in U_3}|T_{a_1,v_1}+T_{v_1,w_1}+T_{w_1,z}-
(T_{a_1,v_2}+T_{v_2,w_2}+T_{w_2,z})|.\]
If in addition, the configuration on $\gamma_1[x,mx]$ and $\gamma_2[x,mx]$ are such that $\gamma_1,\gamma_2$ are \emph{typical}, then on $D_2$ one has, using \eqref{e:expect},
\begin{equation}\label{e:difftyp}
\sup_{\gamma_1,\gamma_2 \emph{ typical }}|T_{\gamma_1,x,mx}-T_{\gamma_2,x,mx}|\leq cM^2x^{1/3}+6H\sqrt{M}x^{1/3}.
\end{equation}

Next we use Lemma \ref{l:Agamma} to show that $\bigcap_{\gamma \emph{ typical}}\{F^1_\gamma<T_{\gamma,x,mx}-\sqrt{S}x^{1/3}\}$
is a high probability event. Recall the region $B_0$ defined in \eqref{staysinbarrier}. It is easy to see using standard arguments that the geodesic from $(x,(2M+S)x^{2/3})$ to $a_2$ stays above the region $B_0$ between $x$ and $mx$ with probability at least $\frac{999}{1000}$. On this event, and on $D_1\cap D_2$, because of \eqref{e:difftyp}, it follows that (recall that $S$ was chosen much larger compared to $H,M$), if $\{F^1_\gamma<T_{\gamma,x,mx}-2\sqrt{S}x^{1/3}\}$ holds for some $\gamma$ \emph{typical}, then, $\bigcap_{\gamma \emph{ typical}}\{F^1_\gamma<T_{\gamma,x,mx}-\sqrt{S}x^{1/3}\}$ holds. Hence using Lemma \ref{l:Agamma},
\begin{equation}\label{e:allAgamma}
\P\Big(\bigcap_{\gamma \emph{ typical}}\{F^1_\gamma<T_{\gamma,x,mx}-\sqrt{S}x^{1/3}\}\Big)\geq \frac{98}{100}.
\end{equation}

Now we show that $\P(A^1_{\Gamma_0'})\geq \frac{97}{100}$. The bounds for $A^2_{\Gamma_0'}, A^3_{\Gamma_0'}$ follow similarly and that for $A_{\Gamma_0'}$ follows by taking a union bound. It is easy to see that,
\[\P(A^1_{\Gamma_0'})\geq \P\Big(\bigcap_{\gamma \emph{ typical}}\{F^1_\gamma<T_{\gamma,x,mx}-\sqrt{S}x^{1/3}\}\cap \{\Gamma_0' \mbox{ is typical}\}\Big)\]
Since $\P(\Gamma_0' \mbox{ is typical})\geq \frac{99}{100}$ by Lemma \ref{typicalpath}, it follows by using \eqref{e:allAgamma} and taking a union bound that
\[\P(A^1_{\Gamma_0'})\geq \frac{97}{100},\]
completing the proof.
\end{proof}

Let $\Gamma_0$ be the geodesic from $a_1$ to $a_2$.
The crux of the next lemma is that on the event that $G, R_{\gamma}$ and $A_\gamma$ occur and $\gamma$ is \emph{typical}, $\Gamma_0$ merges with $\gamma$.

\begin{lemma}
\label{intersectmeet}
If $\gamma$ is any fixed increasing path from $b_1$ to $b_2$, then on the event $G\cap R_{\gamma}\cap A_\gamma\cap \{\gamma \mbox{ is typical }\}$, $\Gamma_0$, the geodesic from $a_1$ to $a_2$, meets $\gamma$.
\end{lemma}
\begin{proof}First observe that if $\Gamma_0$ gets below $\gamma$ at any point, it has to intersect $\gamma$. Also note that, on $A_\gamma\cap \{\gamma \mbox { is typical }\}$, the maximal path $\Gamma_0$ cannot hit the barrier above $(2M+S)x^{2/3}$ distance from the diagonal at $x$ or above $(2M+S)(mx)^{2/3}$ distance from the diagonal at $mx$ without hitting $\gamma$. Also $\gamma$ is a typical path and hence hits the walls of $B$, if $\Gamma_0(x)$ or $\Gamma_0(mx)$ is below the walls of $B$, it must already intersect $\gamma$. Otherwise, $\Gamma_0$ enters and exits through the left and right walls of $B$. We show that this cannot happen without hitting $\gamma$. See Figure \ref{f:barrier}.

\begin{figure}[htb!]
\centering
\includegraphics[width=0.3\textwidth]{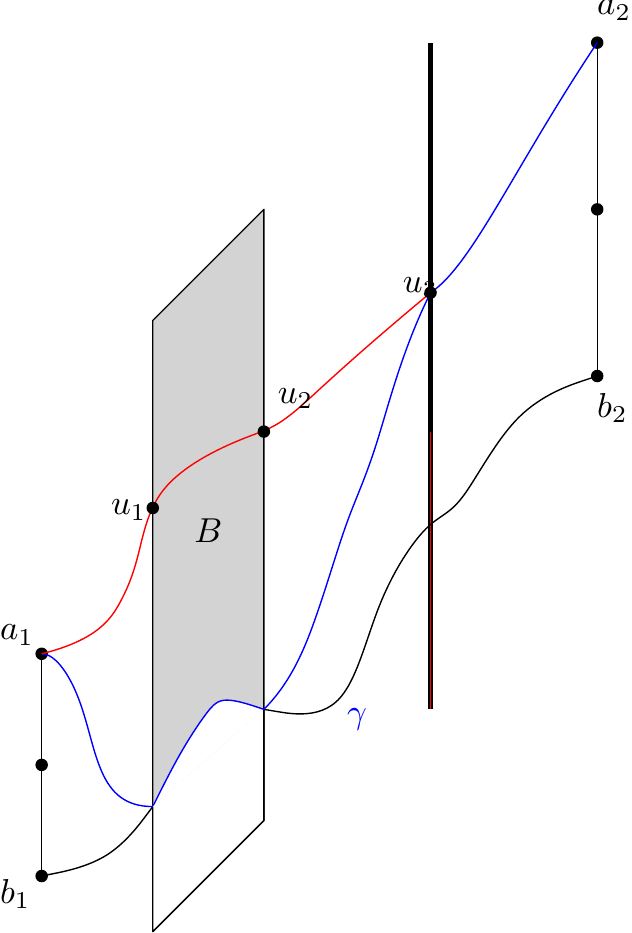}
\caption{Merging of $\Gamma_0$ with $\gamma$ in the proof of Lemma \ref{intersectmeet}. On $A_{\gamma}$, any path that passes above the barrier or the line $L_3$ is uncompetitive, and if any path crosses the barrier and hits the line $L_3$, the red part of the path can be replaced by the blue path which on the good events and presence of the barrier has larger weight and also merges with $\gamma$.}
\label{f:barrier}
\end{figure}

Let the point on the left wall of $B$ where $\Gamma_0$ enters the barrier be $u_1$ and the point on the right wall of $B$ where $\Gamma_0$  exits the barrier be $u_2$. Also the point where $\Gamma_0$ intersects $L_3$ be $u_3$ ($\Gamma_0$ must intersect $L_3$ since on $A_\gamma$ any path passing above $L_3$ is worse than a path that merges with $\gamma$, and if it passes below $L_3$, it intersects with $\gamma$, as $\gamma$ passes through $L_3$). We compare the part of the path $\Gamma_0$ till $L_3$ with $\Gamma_{a_1,(x,\gamma(x))}\cup \gamma[x,mx]\cup \Gamma_{(mx,\gamma(mx)), u_3}$ (by a minor abuse of notation we denote by $\gamma[x,mx]$ the part of $\gamma$ between $(x,\gamma(x))$ and $(mx,\gamma(mx))$, it does not affect our calculations in any way). Hence enough to prove 
\begin{equation}\label{comparecase1}
\widehat{T}_{a_1,u_1}+\widehat{T}^\gamma_{u_1,u_2}+\widehat{T}_{u_2,u_3}\leq
\widehat{T}_{a_1,(x,\gamma(x))}+\widehat{\l(\gamma[x,mx])}+\widehat{T}_{(mx,\gamma(mx)),u_3}.
\end{equation}
This follows because on $G\cap R_{\gamma}$, for $\gamma$ typical,
\[\widehat{T}_{a_1,u_1}\leq S^3x^{1/3},\widehat{T}^\gamma_{u_1,u_2}\leq -S^4x^{1/3}, \widehat{T}_{u_2,u_3}\leq S^3x^{1/3}, \]
\[\widehat{T}_{a_1,(x,\gamma(x))}\geq -S^3x^{1/3},\widehat{\l(\gamma[x,mx])}\geq -H\sqrt{M}x^{1/3},\widehat{T}_{(mx,\gamma(mx)),u_3}\geq -S^3x^{1/3}. \]
Here we have used that since $u_1<(x,x+2Sx^{2/3})$, hence for some absolute constant $c$, $|\E T_{a_1,u_1}-2d(a_1,u_1)|\leq cS^2x^{1/3}$. This, together with the fact that $|\widetilde{T}_{a_1,u_1}|\leq H\sqrt{S}x^{1/3}$ because of the event $G$, one has that $|\widehat{T}_{a_1,u_1}|\leq S^3x^{1/3}$ by choosing $S$ large. Similar arguments apply to $\widehat{T}_{u_2,u_3}, \widehat{T}_{a_1,(x,\gamma(x))}$ and $\widehat{T}_{(mx,\gamma(mx)),u_3}$.

From here equation \eqref{comparecase1} follows by choosing $S$ sufficiently large compared to $M$ and $H$.
\end{proof}

\subsection{Proof of Proposition \ref{l:meetsubinter} }
\label{s:finish}

In this subsection we complete the proof of Proposition \ref{l:meetsubinter}. As stated earlier, without loss of generality, we shall prove it for $r=1$.

\begin{proof}[Proof of Proposition \ref{l:meetsubinter}] Consider the events $G, R_{\gamma}$ and the barrier $B$ and the event $A_{\Gamma_0'}$ as defined in the previous sections.
The proof shall follow by conditioning on the lower path $\Gamma'_0=\gamma$. We first define sets $J_1, J_2, J_3$ of increasing paths $\gamma$ from $b_1$ to $b_2$ together with the configuration $\Pi_{\{\gamma\}}$ on $\{\gamma\}$ such that it is very likely that $\Gamma_0'\in J_1\cap J_2\cap J_3$.

Let $J_1$ denote the set of all \emph{typical} paths from $b_1$ to $b_2$. Then Lemma \ref{typicalpath} gives that $\P(\Gamma_0'\in J_1)\geq \frac{9}{10}$.

Let $J_2$ denote the set of all increasing paths $\gamma$ from $b_1$ to $b_2$  and configurations $\Pi_{\{\gamma\}}$ such that $\P(G|\Gamma_0'=\gamma)\geq \frac{9}{10}$.
Since $\P(G)\geq \frac{99}{100}$, one gets by Markov's inequality,
\[\P(\Gamma_0'\in J_2)\geq \frac{9}{10}.\]

Let $J_3$ denote the set of all increasing paths $\gamma$ from $b_1$ to $b_2$ together with the configurations $\Pi_{\{\gamma\}}$ such that $\P(A_{\Gamma_0'}|\Gamma_0'=\gamma)\geq \frac{2}{3}$. Since by Lemma \ref{l:AGam} $\P(A_{\Gamma_0'})\geq \frac{9}{10}$, by Markov's inequality,
\[\P(\Gamma_0'\in J_3)\geq \frac{7}{10}.\]
Then by union bound, $\P(\Gamma_0'\in J_1\cap J_2\cap J_3)\geq \frac{1}{2}$.

Fix a particular $(\gamma,\Pi_{\gamma}) \in J_1\cap J_2\cap J_3$.

Since
\[\P(A_\gamma|\Gamma_0'=\gamma)=\P(A_{\Gamma_0'}|\Gamma_0'=\gamma)\geq \frac{2}{3},\]
and $\P(G|\Gamma_0'=\gamma)\geq \frac{9}{10}$, hence
\begin{equation}\label{e:GA}
\P(G\cap A_\gamma|\Gamma_0'=\gamma) \geq \frac{1}{2}.
\end{equation}

Also as $R_{\gamma}$ is a decreasing event on the configuration of $\cZ \setminus\{\gamma\}$, and is independent of the configuration in $(\cZ\setminus\{\gamma\})^c$, and $A_\gamma$ and $\Gamma_0'=\gamma$ are also decreasing in the configuration of $\cZ\setminus\{\gamma\}$, by FKG inequality it follows that,
\[\P(R_\gamma\cap A_\gamma \cap \{\Gamma_0'=\gamma\}|(\cZ\setminus{\{\gamma\}})^c)\geq \P(R_\gamma)\P(A_\gamma \cap \{\Gamma_0'=\gamma\}|(\cZ\setminus{\{\gamma\}})^c).\]
As $G$ is $(\cZ)^c$ measurable,
\[\P(R_\gamma\cap A_\gamma \cap \{\Gamma_0'=\gamma\}|G)\geq \P(R_\gamma)\P(A_\gamma \cap \{\Gamma_0'=\gamma\}|G).\]
Hence,
\[\P(R_{\gamma}|G\cap A_\gamma \cap \{\Gamma_0'=\gamma\})\geq\P(R_{\gamma})\geq \beta>0.\]
This, together with \eqref{e:GA}, gives
\[\P(R_{\gamma}\cap G\cap A_\gamma|\Gamma_0'=\gamma)\geq \frac{\beta}{2}=:\beta'>0.\]

Now, it follows from Lemma \ref{intersectmeet} that on the event $R_{\gamma}\cap G\cap A_{\gamma}\cap \{\gamma \mbox{ is typical }\}$, $\Gamma_0$ meets $\gamma$. Hence for any fixed $(\gamma,\Pi_{\gamma}) \in J_1\cap J_2\cap J_3$,
\[
\P(\Gamma_0 \mbox{ meets } \gamma|\Gamma_0'=\gamma)\geq \P(R_{\gamma}\cap G\cap A_\gamma|\Gamma_0'=\gamma)\geq \beta'>0,
\]
where $\beta'=\frac{\beta}{2}$ is an absolute positive constant not depending on the typical path $\gamma$. Hence, by integrating over all $(\gamma,\Pi_{\gamma})\in J_1\cap J_2\cap J_3$,
\[\P(\Gamma_0 \mbox{ meets } \Gamma_0')\geq \beta'\P(\Gamma_0' \in J_1\cap J_2\cap J_3)\geq \beta'/2=:\alpha>0.\]
Also observe that $\alpha$ does not depend on $z$, this completes the proof. 
\end{proof}


%
%

\section{Optimal tail estimate for coalescence of semi-infinite geodesics}
\label{s:coalopt}
 


In this section we prove Theorem \ref{t:coalopt} for the semi-infinite geodesics. Before proceeding with the proof let us briefly recall the setting of the theorem. We had points $v_3=(k^{2/3},-k^{2/3})$ and $v_4=(-k^{2/3},k^{2/3})$, and we denoted by $\Gamma_{v_3}$ and $\Gamma_{v_4}$ the semi-infinite geodesics started respectively from $v_3$ and $v_4$ in the direction $(1,1)$, and $v^*=(v^*_1,v^*_2)$ denoted the point of coalescence of $v_3$ and $v_4$. The distance to coalescence $d(v_3,v_4)$ was defined to be equal to $v^*_1+v^*_2$. As mentioned before to prove that $\P(d(v_3,v_4)>Rk)\approx CR^{-2/3}$ we shall appeal to the translation invariance of the underlying passage time field. 

Observe that it suffices to prove Theorem \ref{t:coalopt} for all sufficiently large $k$. Fix now $k$ sufficiently large. Let us now identify the line $\mathbb{L}$ with $\Z$ via the identification $i\mapsto u_i=\mathbf{0}+i(-1,1)$. We call $u_i\in \mathbb{L}$ a $k$-\textbf{boundary point} if $d(u_i,u_{i+1})>k$; see Figure \ref{f:boundary}. Define a sequence $\{X^{(k)}_{i}\}_{i\in \Z}$ by setting $X_{i}=1$ if $u_i$ is a $k$-boundary point and $0$ otherwise. Observe that translation invariance implies that this is a stationary sequence. The main step of the proof will be the following proposition.

\begin{figure}[htb!]
\centering
\includegraphics[width=0.5\textwidth]{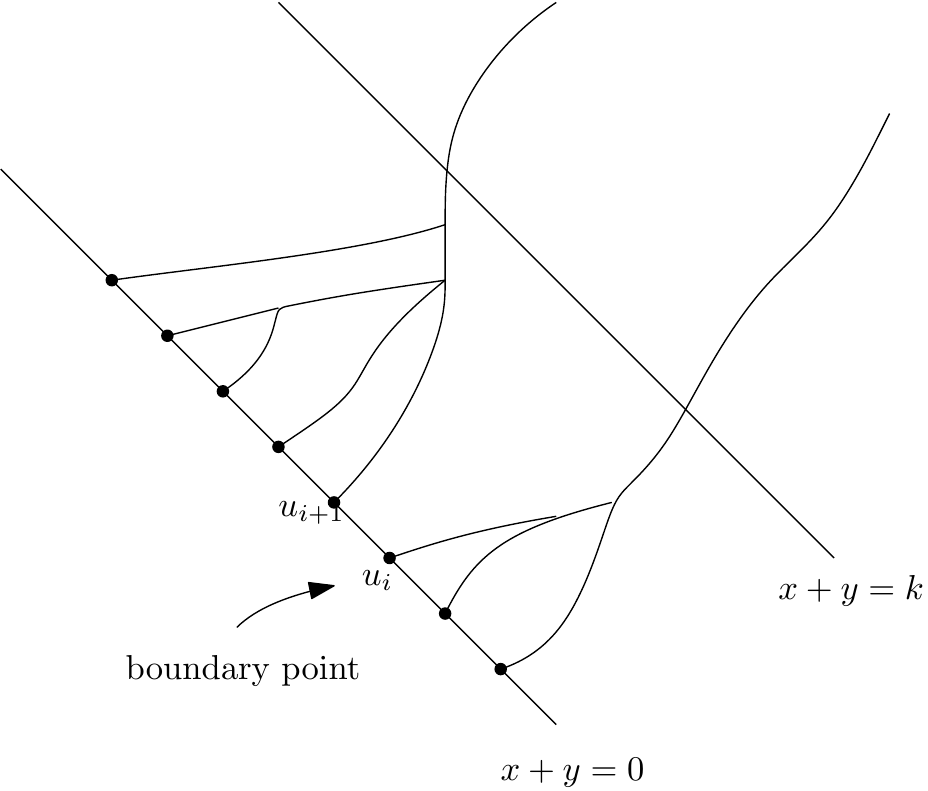}
\caption{Definition of a $k$-boundary point. If semi-infinite geodesic starting from two neighbouring points on $x+y=0$ coalesce above the line $x+y=k$ then one of them is a $k$-boundary point.}
\label{f:boundary}
\end{figure}

\begin{proposition}
\label{p:upperlower}
There exists $C_1,C_2>0$ such that for each $i\in \Z$ and for each $k$, we have 
$$\frac{C_2}{k^{2/3}}\leq \P(X^{(k)}_{i}=1)\leq \frac{C_1}{k^{2/3}}.$$
\end{proposition}

We postpone the proof of Proposition \ref{p:upperlower} for now and prove the upper bound of Theorem \ref{t:coalopt} first, the lower bound of Theorem \ref{t:coalopt} will be proved at the end of this section. 

\begin{proof}[Proof of Theorem \ref{t:coalopt}: upper bound]
Fix $R>1$. Clearly, on $\{d(v_3,v_4)>Rk\}$, there must exist $i\in \llbracket -k^{2/3}, k^{2/3}\rrbracket$ such that $u_{i}$ is an $Rk$-boundary point. It follows that
$$\P(d(v_3,v_4)>Rk)=\P\left(\sum_{i=-k^{2/3}}^{k^{2/3}} X^{(Rk)}_{i}>0\right)\leq \E\left[\sum_{i=-k^{2/3}}^{k^{2/3}} X^{(Rk)}_{i}\right] \leq 2k^{2/3}\frac{C_1}{(Rk)^{2/3}}$$
where the final inequality follows from Proposition \ref{p:upperlower}. This completes the proof of the theorem.
\end{proof}

\subsection{Proof of Proposition \ref{p:upperlower}}
We prove Proposition \ref{p:upperlower} in this subsection. This proof is essentially independent of the rest of the paper, except we need to use a variant of Theorem \ref{t:carsestimate}, that also holds for the semi-infinite geodesics. We first record this statement. 

\begin{proposition}
\label{p:carinfinite}
For $v$ on the line $\mathbb{L}:x+y=0$, let $f(v)=(f(v)_1,f(v)_2)$ denote the point the semi-infinite geodesic $\Gamma_{v}$ started from $v$ in the direction $(1,1)$ intersects the line $x+y=k$. For $h>0$, let $A_{h}=A_{h,k}$ denote the event that there exists a point $v$ on $\mathbb{L}$ between $v_3$ and $v_4$ such that $|f(v)_1-f(v)_2|\geq hk^{2/3}$. Then there exists $h_0>0,c>0$ such that for all $h>h_0$ and all sufficiently large $k$ we have 
$\P(A_{h})\leq e^{-ch^2}$.
\end{proposition}

{We shall not provide a detailed proof of this proposition, but let us indicate how one can obtain this result by arguing as in the proof of Theorem \ref{t:carsestimate}. Without loss of generality take $v=\mathbf{0}$. We need to upper bound $\P(f(0)_1-f(0)_{2}\geq hk^{2/3})$. Let $\mathcal{L}$ denote the straight line $y=(1-k^{-1/3})x$. By definition of $\Gamma_{\mathbf{0}}$, all but finitely many points on it lies to the left of $\mathcal{L}$, while $f(v)$ lies to its right (for $h$ sufficiently large) in the event $f(0)_1-f(0)_{2}\geq hk^{2/3}$. We can now use the strategy of Proof of Theorem \ref{t:carsestimate} to show that it is unlikely for the path to cross the line $\mathcal{L}$, by checking at dyadically increasing scales. Here we need to use the observation that the proof of Theorem \ref{t:carsestimate} works for paths in the direction other than $(1,1)$ as explained in Corollary \ref{carsestimate2}. We omit the details.}


The main input of this section however is the following from \cite{BHS18}, a result obtained by the first and third author jointly with Christopher Hoffman. A slightly stronger version of the result is used in \cite{BHS18} to show the non-existence of bigeodesics in exponential LPP. We need some preparation to describe the result. Let $\mathbb{L}_{0}$ denote the line segment joining $(k^{2/3},-k^{2/3})$ and $(-k^{2/3},k^{2/3})$. For $h\in \N$, let $\mathbb{L}'_{h}$ denote the line segment $(k,k)+2h\mathbb{L}_0$. For points $u,v\in \mathbb{L}_{0}$ we say $u\leq v$ if $v=u+i(-1,1)$ for some $i>0$, and similarly on $\mathbb{L}'_{h}$. For $\ell\in \N$, let $\mathcal{C}_{\ell,h}$ denote the event that  there exists points $u_1\leq u_2\leq \cdots \leq u_{\ell}$ on $\mathbb{L}_0$, and $w_1\leq w_2\leq \cdots \leq w_{\ell}$ on $\mathbb{L}'_{h}$ such that the geodesics $\Gamma_{u_i,w_i}$ are disjoint. We quote the following theorem from \cite{BHS18}.

\begin{proposition}\cite[Corollary 2.8]{BHS18}
\label{p:rare}
There exists $k_0,\ell_0>0$, such that for all $k>k_0, k^{0.01}>\ell>\ell_0$ and all $h\leq \ell^{1/16}$ we have 
$\P(\mathcal{C}_{\ell,h})\leq e^{-c\ell^{1/4}}$
for some positive constant $c$. 
\end{proposition}

Observe that this proposition immediately implies that in the same set-up $\P(\mathcal{C}_{\ell,h})\leq e^{-c\ell^{1/4}}+e^{-ck^{0.001}}$ for all $\ell>\ell_0$ and $h\leq k^{0.0005}$.  The main idea of the proof of Proposition \ref{p:rare} is the following: if there are too many disjoint geodesics across a rectangle of size $k\times k^{2/3}$, there must be one which is constrained to be in a thin rectangle. The proof can be completed with using the by now well-known fact that paths restricted to thin rectangles are unlikely to be competitive \cite{BGH17, BG18}, together with an application of BK inequality. See \cite{BHS18} for the detailed argument, we shall omit this proof. 

We shall now complete the proof of Proposition \ref{p:upperlower} using Proposition \ref{p:rare}. First we need the following lemma.

\begin{lemma}
\label{l:sum}
There exists $k_0,l_0$ such that for all $k>k_0$, and $\ell>\ell_0$, we have 
$$\P\left(\sum_{i=-k^{2/3}}^{k^{2/3}} X^{(k)}_{i} \geq \ell \right) \leq e^{-c\ell^{1/8}}+e^{-ck^{0.001}}.$$
\end{lemma}

\begin{proof}
Fix $\ell$ sufficiently large and set $h=\min\{{\ell}^{1/16}, k^{0.0005}\}$. Now observe that on the event 
$$\left\{\sum_{i=-k^{2/3}}^{k^{2/3}} X^{(k)}_{i} \geq \ell \right\}$$ we must either have that the event $A_{h}$ from Proposition \ref{p:carinfinite} holds, or the event $C_{{\ell},h}$ from Proposition \ref{p:rare} holds. Observe that, by Proposition \ref{p:carinfinite} the probability of the first event is bounded by $e^{-c{\ell}^{1/8}}+e^{-ck^{0.001}}$, and by Proposition \ref{p:rare} the probability of the second event is also bounded by $e^{-c{\ell}^{1/8}}+e^{-ck^{0.001}}$. The proof is completed by taking a union bound. 
\end{proof}

We are now ready to prove Proposition \ref{p:upperlower}.

\begin{proof}[Proof of Proposition \ref{p:upperlower}]
For the upper bound simply note that by translation invariance $\P(X^{(k)}_{i}=1)$ is independent of $i$. The proof is now completed by noting that $\sum_{i=-k^{2/3}}^{k^{2/3}} X^{(k)}_{i}\leq 2k^{2/3}$ and Lemma \ref{l:sum} implies that  $\E \sum_{i=-k^{2/3}}^{k^{2/3}} X^{(k)}_{i} \leq C$ for some large constant $C$ uniformly in all large $k$.

For the lower bound, observe the following. By the same argument as in the first part, it suffices to prove that there exists $M$ sufficiently large such that with $w_1=(Mk^{2/3},-Mk^{2/3})$ and $w_2=(-Mk^{2/3},Mk^{2/3})$, we have for all $k>0$ 
\begin{equation}
\label{e:lb}
\P(d(w_1,w_2)>k)\geq \frac{1}{3},
\end{equation}
which ensures that with probability bounded away from $0$, there is at least one boundary point out of the $2Mk^{2/3}$ between $w_1$ and $w_2$.
The existence of such an $M$ follows from Proposition \ref{p:carinfinite} by noticing that for $M$ sufficiently large the semi-infinite geodesic from $w_1$ in the direction $(1,1)$ hits the line $x+y=k$ below the point $(\frac{k}{2}, \frac{k}{2})$ with probability at least $2/3$, whereas the semi-infinite geodesic from $w_1$ in the direction $(1,1)$ hits the line $x+y=k$ above the point $(\frac{k}{2}, \frac{k}{2})$ with probability at least $2/3$.
\end{proof}

\begin{remark}
\label{r:lb}
Observe that using the same argument as in the proof of the lower bound in the above proposition, it follows (with notations as in Theorem \ref{t:coalopt}) that one has $\limsup_{k\to \infty} \P(d(v_3,v_4)\leq Rk) \leq e^{-c/R}$ for some constant $c>0$, for $R$ small. It recovers the lower bound on distance to coalescence obtained by \cite{Pim16}, with a better quantitative estimate.  
\end{remark}

\subsection{Lower Bound in Theorem \ref{t:coalopt}}
It remains to prove the lower bound in Theorem \ref{t:coalopt}. As mentioned before, this has already been proved by Pimentel \cite{Pim16, Pim17} using a duality formula but we provide a short alternative proof using our techniques.



%
%


\begin{proof}[Proof of Theorem \ref{t:coalopt}: lower bound]
Starting as in the proof of the upper bound, with same notations, we are required to lower bound $\P\left(\sum_{i=-k^{2/3}}^{k^{2/3}} X^{(Rk)}_{i}>0\right)$. Let $M$ be as in \eqref{e:lb}. By translation invariance it follows that 
$$MR^{2/3}\P\left(\sum_{i=-k^{2/3}}^{k^{2/3}} X^{(Rk)}_{i}>0\right) \geq \P(d(w'_1,w'_2)>Rk)>\frac{1}{3}$$
where $w'_1=(M(Rk)^{2/3},-M(Rk)^{2/3})$ and $w'_2=(-M(Rk)^{2/3},M(Rk)^{2/3})$. The lower bound $\P\left(\sum_{i=-k^{2/3}}^{k^{2/3}} X^{(Rk)}_{i}>0\right) \geq \frac{1}{3M}R^{-2/3}$ is now immediate.
\end{proof}

An alternative proof of the above can also be obtained by using second moment method.

\bibliography{slowbond}
\bibliographystyle{plain}

\Addresses

\end{document}